\documentclass[12pt, reqno]{amsart}
\allowdisplaybreaks[1]
\usepackage{amsmath}
\usepackage{amssymb}
\usepackage{amsfonts}
\usepackage{verbatim}
\usepackage[usenames]{color}
\usepackage{hyperref}
 \newtheorem{theorem}{Theorem}[section]
 \newtheorem{corollary}[theorem]{Corollary}
 \newtheorem{lemma}[theorem]{Lemma}
 \newtheorem{proposition}[theorem]{Proposition}

\theoremstyle{definition}
\newtheorem{definition}[theorem]{Definition}
\theoremstyle{remark}

\newtheorem{fact*}{Fact}
\DeclareMathOperator{\RE}{Re}
\DeclareMathOperator{\IM}{Im}

\newcommand\dd{\mathrm d}

\newcommand{\hilbert}{\mathcal{H}}

\newcommand{\B}{\mathcal{B}}

\newcommand{\vecspace}{\mathcal{V}}

\newcommand{\K}{\mathcal{K}}
\newcommand{\M}{\mathcal{M}}
\newcommand{\N}{\mathcal{N}}

\newcommand{\D}{\mathbb{D}}
\newcommand{\C}{\mathbb{C}}

\newcommand{\R}{\mathbb{R}}

\newcommand{\cc}[1]{\overline{#1}}
\newcommand{\abs}[1]{\left\vert#1\right\vert}

\newcommand{\norm}[1]{\left\Vert#1\right\Vert}

\newcommand{\ran}[1]{\operatorname{ran}#1}

\newcommand{\ip}[2]{\left\langle #1, #2 \right\rangle}
\newcommand{\ad}{^\ast}
\newcommand{\inv}{^{-1}}

\newcommand{\til}{\raise.17ex\hbox{$\scriptstyle\mathtt{\sim}$}}

\newcommand\Pick{\mathcal P}

\newcommand{\ph}{\varphi}

\newcommand\la{\lambda}

\newcommand\beq{\begin{equation}}

\newcommand\eeq{\end{equation}}

\newcommand\RR{\mathcal{R}}
\newcommand{\vectwo}[2]
{
   \begin{pmatrix} #1 \\ #2 \end{pmatrix}
}

\newcommand{\bbm}{\left[ \begin{smallmatrix}}
\newcommand{\ebm}{\end{smallmatrix} \right]}
\newcommand{\bpm}{\left( \begin{smallmatrix}}
\newcommand{\epm}{\end{smallmatrix} \right)}
\numberwithin{equation}{section}

\newcommand{\tensor}[2]{\text{ }{\begin{smallmatrix} #1 \\ \otimes\\ #2\end{smallmatrix}}\text{  }}
\newcommand{\flattensor}[2]{#1 \otimes #2}
\newcommand{\bigtensor}[2]{\text{ }\begin{matrix} #1 \\{\otimes}\\ #2\end{matrix}\text{  }}
\newlength{\Mheight}
\newlength{\cwidth}
\newcommand{\mc}{\settoheight{\Mheight}{M}\settowidth{\cwidth}{c}M\parbox[b][\Mheight][t]{\cwidth}{c}}

\title[Free Pick functions]{Free Pick functions: Representations, asymptotic behavior and matrix monotonicity in several noncommuting variables}
\author{
J. E. Pascoe \\ \\
Ryan Tully-Doyle
}
\date{\today}

\begin{document}

\begin{abstract}
	We extend the study of the Pick class, the set of complex analytic
	functions taking the upper half plane into itself,
	to the noncommutative setting. R. Nevanlinna showed that elements of the Pick class have
	certain integral representations which reflect their asymptotic behavior at infinity.
	L\"owner connected the Pick class to matrix monotone functions.
	We generalize the Nevanlinna representation theorems and L\"owner's theorem on
	matrix monotone functions to the \emph{free Pick class}, the collection of functions that map tuples of matrices with positive imaginary
	part into the matrices with positive imaginary part which
	obey the free functional calculus.
	
	%We study the \emph{free Pick class}, i.e the collection of functions that map tuples of matrices with positive imaginary
	%part into the matrices with positive imaginary part which
	%obey the free functional calculus.
	%Free Pick functions are shown to have Nevanlinna representations
	%which reflect asymptotic behavior at infinity.
	%L\"owner's theorem is proved:
	%multivariable matrix monotone functions
	%analytically continue to functions
	%in the free Pick class.
	%We give a sum of squares representation for the derivative of a multivariable matrix monotone function
	%following the example of noncommutative real algebraic geometry.
	%A principal tool is a generalization of the classical Hardy space to noncommuting variables, which gives
	%a canonical sum of squares representation pointwise which can be patched together algebraically to obtain global results.
\end{abstract}
\maketitle

\tableofcontents

\section{Introduction}

%\subsection{The classical theory of Pick functions}
Let $\mathbb{H}\subset \mathbb{C}$ denote the complex upper half plane. That is,
	$$\mathbb{H} = \{z \in \mathbb{C} | \IM z > 0 \}.$$
The \emph{Pick class} is the set of analytic functions $f: \mathbb{H} \rightarrow \mathbb{H}.$ The elements of the Pick class
are called \emph{Pick functions.}
%Mathematicians of the early 20th century developed the theory within the context of modern analysis \cite{ju20, car29, nev22}.

Rolf Nevanlinna showed that a subset of the Pick class satisfying an
asymptotic condition at infinity is exactly parametrized by positive Borel measures on the real line.
	\begin{theorem}[R. Nevanlinna \cite{nev22}] \label{NevanlinnaRep}
	Let $h: \mathbb{H} \to \C$. There exists a finite Borel positive measure $\mu$ on $\R$ such that 
	\beq\label{NevanlinnaRepFormula}
	h(z) = \int\!\frac{1}{t-z}d\mu(t)
	\eeq
	if and only if $h$ is in the Pick class and
	\beq\label{NevanlinnaRepFormulaAsy}
	\liminf_{s\to\infty} s\abs{h(is)} < \infty.
	\eeq
	Moreover, for any Pick function $h$ satisfying Equation \ref{NevanlinnaRepFormulaAsy}
	the measure $\mu$ in Equation \ref{NevanlinnaRepFormula} is uniquely determined.
	\end{theorem}
The map taking a finite positive Borel measure $\mu$ to a Pick function $h$ via
the correspondence in Equation \eqref{NevanlinnaRepFormula} is called the \emph{Cauchy transform.}
A na\"ive interpretation of the asymptotic condition at infinity, Equation \eqref{NevanlinnaRepFormulaAsy}, is that the first residue exists. This was studied rigorously in the guise of a conformally equivalent condition on self maps of the unit disk by Julia \cite{ju20} and Carath\'eodory \cite{car29}. 

Nevanlinna applied Theorem \ref{NevanlinnaRep} to the Hamburger moment problem:
Given a sequence of real numbers $(\rho_i)^{\infty}_{i=0}$, when does there exist a
measure $\mu$ such that for each $i\in \mathbb N,$ $\rho_i$ is the $i$-th moment of the measure $\mu,$
i.e.,
	$$\rho_i = \int x^i d\mu?$$
\begin{theorem}[R. Nevanlinna \cite{nev22}] 
	Let $(\rho_i)^{\infty}_{i=1}$ be a sequence of real numbers. The following are equivalent.
	\begin{enumerate}
		\item There is a finite positive Borel measure $\mu$ on $\R$ so that, for each $i\in \mathbb N,$
			$$\rho_i = \int x^i d\mu.$$
		\item There is a Pick function $h$ such that, for every $N\in \mathbb N,$
			$$h(z) = \sum^{N}_{i=0} \frac{1}{z^{i+1}}\rho_i + O(\frac{1}{|z|^{N+2}}).$$
		\item The infinite matrix Hankel matrix
			$$A = [\rho_{i+j}]_{0 \leq i,j \leq \infty} = \bbm \rho_0 & \rho_1 & \rho_2 & \ldots\\
			\rho_1 & \rho_2 & \rho_3 & \ldots \\
			\rho_2 & \rho_3 & \rho_4 & \ldots \\
			\vdots & \vdots & \vdots & \ddots 
			\ebm$$
		is positive semidefinite in the sense that, for each in $N \in \mathbb N,$
		the truncated Hankel matrix $[\rho_{i+j}]_{0 \leq i,j\leq N}$ is positive semidefinite.	
	\end{enumerate}
\end{theorem}

Pick functions also correspond to matrix monotone functions via L\"owner's theorem.
Given a function $f:(a,b) \rightarrow \mathbb{R},$ we extend $f$ via the functional calculus to self-adjoint matrices $A$
with spectrum in $(a,b)$ by taking the diagonalization of $A$ by a unitary matrix $U,$ that is,
	$$A = U^* \bbm \lambda_1 & & \\
	 & \lambda_2 & \\
	 & & \ddots \\
	\ebm U,$$
and defining
	\beq\label{BFC}
	f(A) = U^* \bbm f(\lambda_1) & & \\
	 & f(\lambda_2) & \\
	 & & \ddots \\
	\ebm U.
	\eeq
A function $f:(a,b) \rightarrow \mathbb{R}$ is called \emph{matrix monotone} if
	$$A \leq B \Rightarrow f(A) \leq f(B)$$
where $A \leq B$ means that $B-A$ is positive semidefinite.

The condition that a function $f:(a,b) \rightarrow \mathbb{R}$ be matrix monotone is much stronger than that
$f$ should be monotone in the ordinary sense. For example, let the function $f:\mathbb{R} \rightarrow \mathbb{R}$ be given by the formula
	$$f(x) = x^3.$$
The function $f$ is monotone on all of $\mathbb{R}.$
Note that
	$$\bpm 1 & 1 \\ 1 & 1 \epm \leq \bpm 2 & 1 \\ 1 & 1 \epm$$
since   $$\bpm 2 & 1 \\ 1 & 1 \epm - \bpm 1 & 1 \\ 1 & 1 \epm = \bpm 1 & 0 \\ 0 & 0 \epm$$ is a positive semidefinite matrix.
However, 
	$$\begin{matrix}
	f\bpm 1 & 1 \\ 1 & 1 \epm & = & \bpm 4 & 4 \\ 4 & 4 \epm \\
	f\bpm 2 & 1 \\ 1 & 1 \epm & = & \bpm 13 & 8 \\ 8 & 5 \epm \\
	\end{matrix}$$
and 
	$$\bpm 13 & 8 \\ 8 & 5 \epm - \bpm 4 & 4 \\ 4 & 4 \epm = \bpm 9 & 4 \\ 4 & 1 \epm$$
which is not positive semidefinite since $\text{det} \bpm 9 & 4 \\ 4 & 1 \epm = -5 < 0$, and so $f(x) = x^3$ is not matrix monotone even though it is monotone on all of $\R$.

In \cite{lo34}, Charles L\"owner showed the following theorem.
\begin{theorem}[L\"owner \cite{lo34}]\label{LownerOne}
Let $f: (a,b) \to \R$ be a bounded Borel function. If $f$ is matrix monotone, then $f$ analytically continues to the upper half plane
as a function in the Pick class.
\end{theorem}
For a modern treatment of L\"owner's theorem, see e.g. \cite{don74,bha97,bha07}.

L\"owner's theorem can be used to identify whether or not many classically important functions are matrix monotone. For example,
$x^{1/3}, \log x,$ and $-\frac{1}{x}$ are matrix monotone on the interval $(1,2),$ but $x^3$ and $e^x$ are not.

Interpreting Theorem \ref{LownerOne} in the context of Nevanlinna's solution to the Hamburger moment problem, we obtain the following corollary, which will guide our study of matrix monotone functions in several variables.
\begin{corollary}\label{seriesCharacter}
Let  $f(x) = \sum^{\infty}_{i=0} a_ix^i$ be a power series which converges on a neighborhood of the closed disk $\overline{\mathbb{D}}$. The function $f$ is matrix monotone if and only if the infinite Hankel matrix
			$[a_{i+j+1}]_{0 \leq i,j \leq \infty}\geq 0.$
\end{corollary}

Via the connection to moment problems and matrix monotonicity, the theory of Pick functions has deep and well-studied consequences for science and engineering. John von Neumann and Eugene Wigner applied L\"owner's theorem to the theory of quantum collisions \cite{wigner}.
Other applications include quantum data processing \cite{ahls}, wireless communications \cite{Jorswieck2007f, Boche2004e} and engineering \cite{alcober,osaka}.

We execute the program above in \emph{several noncommuting variables.} We now briefly describe the free functional calculus, which gives a meaningful notion of a noncommutative function.

\subsection{The free functional calculus}
	
	The functional calculus in several variables is less well understood than that given in \eqref{BFC} because it is
	noncommutative. The \emph{free functional calculus} is modeled on the theory of free polynomials evaluated
	on tuples of matrices. For this purpose, free polynomials have three important properties which we will
	now illustrate with an example.
	
	Consider the free polynomial in two variables,
		$$p(X,Y) = XY + 7XYX.$$
	
	First, given two $n$ by $n$ matrices with entries in $\mathbb{C},$
	$X, Y \in \M_n(\mathbb{C}),$ the value of $p$ at the point $(X, Y),$ $p(X,Y)$ is again
	a matrix in $\M_n(\mathbb{C}).$ This says that $p$ is a \emph{graded function.}
	
	Second, for two $n$ by $n$ matrices $X_1, Y_1 \in \M_n(\mathbb{C}),$ and two $m$ by $m$ matrices $X_2, Y_2 \in \M_m(\mathbb{C}),$
	consider the calculation
	\begin{align*}
		p\left(\bpm X_1 & \\ & X_2\epm, \bpm Y_1 & \\ & Y_2\epm\right) & =  \bpm X_1 & \\ & X_2\epm \bpm Y_1 & \\ & Y_2\epm 
		+ 7\bpm X_1 & \\ & X_2\epm \bpm Y_1 & \\ & Y_2\epm \bpm X_1 & \\ & X_2\epm , \\
				     & =  \bpm X_1Y_1 + 7X_1Y_1X_1 & \\ & X_2Y_2 + 7X_2Y_2X_2\epm,\\
				     & =  \bpm p(X_1,Y_1) & \\ & p(X_2,Y_2)\epm .
	\end{align*}
	The identity $$p\left(\bpm X_1 & \\ & X_2\epm, \bpm Y_1 & \\ & Y_2\epm\right)=
	\bpm p(X_1,Y_1) & \\ & p(X_2,Y_2)\epm $$ says that $p$ \emph{respects direct sums.}
	
	Third, given two matrices $X, Y \in \M_n(\mathbb{C}),$ and an invertible matrix consider the following calculation of the value
	of $p(S^{-1}XS,S^{-1}YS):$
		\begin{align*}
		p(S^{-1}XS,S^{-1}YS) & = S^{-1}XSS^{-1}YS + 7S^{-1}XSS^{-1}YSS^{-1}XS , \\
				     & =  S^{-1}XYS + 7S^{-1}XYXS,\\
				     & =  S^{-1}(XY + 7XYX )S,\\
				     & =  S^{-1}p(X,Y)S.
		\end{align*}
	The identity $$p(S^{-1}XS,S^{-1}YS) = S^{-1}p(X,Y)S$$ says that $p$ \emph{respects similarity.}

	These three properties, to be graded, respect direct sums and similarity, constitute the definition of a \emph{free function},
	which we now describe precisely.
	
	Let $\M^d$ denote the \emph{$d$-dimensional matrix universe}, which is defined by the equation
	\[
	\M^d = \bigcup_{n=1}^{\infty} \M_n^d(\C).
	\]
	\begin{definition}
	A set $D \subset \M^d$ is called a \emph{free set} if satisfies the following conditions.
	\begin{description}
	\item[$D$ is closed with respect to direct sums] 
	\hfill
	 
	$X = (X_1, \ldots, X_d) \in D$ and $Y = (Y_1, \ldots, Y_d) \in D$ if and only if
	$
	\bpm X & \phantom 0 \\ \phantom 0 & Y \epm = \left(\bpm X_1 & \phantom 0 \\ \phantom 0 & Y_1 \epm, \ldots, \bpm X_d & \phantom 0 \\ \phantom 0 & Y_d \epm\right) \in D.
	$
	\item[$D$ is closed with respect to unitary similarity]
	\hfill

	If $X \in D\bigcap \M_n^d, U \in U_n$, then
	\[
	U\ad X U = (U\ad X_1 U, \ldots, U\ad X_d U) \in D.
	\]
	\end{description}
	Here, $U_n$ denotes the unitary matrices of size $n$.
	\end{definition}
	\begin{definition}
	Let $D$ be a free set.
	Let $f: D \to \M^1$ be a function. We say that $f$ is a \emph{free function} if it satisfies the following conditions.
	\begin{description}
	\item[$f$ is graded] If $X \in D \bigcap \M_n^d$, then $f(X) \in \M_n^1$.
	\item[$f$ respects direct sums] If $X, Y \in D$ then
	\[ 
	f \bpm X & \phantom{0} \\ \phantom 0 & Y \epm = \bpm f(X) & \phantom 0 \\ \phantom 0 & f(Y) \epm.
	\]
	\item[$f$ respects similarity] 
	\hfill

	If $S \in GL_n,$ and $X \in D$ such that $S^{-1}XS \in D$, 
	\[
	f(S\inv X S) = S\inv f(X) S.
	\]
	\end{description}
	Here, $GL_n$ denotes the invertible matricies of size $n$.
	\end{definition}
	
	It is known that for a free set $U$, if $U\cap \M_n$ is open at each level and $f$ is a locally bounded free function, then for 
		$$\hat{U} = \{z |z\in \sigma(A), A \in U\}$$
	there is a unique holomorphic $\hat{f}:\hat{U} \rightarrow \mathbb{C}$ such that $f(A) = \hat{f}(A)$
	for all $A\in U.$
	As in the one variable case, if $U\cap \mathcal{M}^d_n$ is open, then $f$ is
	holomorphic on $U\cap \mathcal{M}^d_n$ as a function of $dn^2$ variables \cite{vvw12}.
	
	We note that free polynomials are free functions. Free rational functions also give free functions off of their singular sets.
	For example,
		$$f(X,Y) = (Y^2 + 5)(1-X)^{-1}Y$$
	is a free function defined on the free set
			$$\{X, Y \in \M^2| 1 \notin \sigma(X)\}.$$
	
	The study of functions in this context is known as \emph{free analysis.}

\subsection{Free Pick functions}

We specify an ordering on tuples of matrices where for $A = (A_1,\ldots, A_d)$ and $B = (B_1,\ldots, B_d)$, the statement
$A> B$ means $A_i-B_i$ is strictly positive definite for each $1 \leq i\leq d$, and $A \geq B$ means $A_i-B_i$ is positive semidefinite  for each $1 \leq i\leq d$.

We consider the \emph{free Pick class} $\mathcal{P}_d,$ the set of free functions defined on the domain
\[
\Pi^d := \{X \in \M^d | \IM X_i = \frac{1}{2}(X_i - X_i\ad) > 0, i = 1\ldots, d\}
\]
with range in $\cc{\Pi^1}$.
These functions directly generalize multivariable Pick functions, maps from the polyupperhalfplane $\Pi_1^d$ to the upper halfplane $\cc{\Pi_1^1}$, as
$$\Pi_1^d = \mathbb{H}^d.$$

Agler, Tully-Doyle, and Young showed the following generalization of Theorem \ref{NevanlinnaRep} \cite{aty13}, presented here in two variables.
\begin{theorem}[Type I representation theorem]\label{oldnevrepstype1}
Let $h$ be a function defined on $\Pi^2_1$. Then there exist a Hilbert space $\hilbert$, a self-adjoint operator $A$ on $\hilbert$, a vector $v \in \hilbert$, and a positive semidefinite contraction $Y$ on $\hilbert$ so that 
\[
h(z) = \ip{(A - z_1 Y - z_2 (1-Y) )\inv v}{v}
\]
if and only if $h$ is a Pick function and
\[
\liminf_{s \to \infty} s\abs{h(is, is)} < \infty.
\]
\end{theorem}

We show that the above theorem also holds for free Pick functions. Our Theorem \ref{newnevrep} contains a precise analogue of Theorem
\ref{oldnevrepstype1} among more general representations. That Theorem \ref{oldnevrepstype1} can be extended in this way is an example of \emph{noncommutative lifting principle}, a guiding principle that states that if a theorem holds by virtue of operator theoretic methods following  Agler's seminal paper of 1990, \emph{On the representation of certain holomorphic functions defined on a polydisc}, then it will hold for free functions. Our proofs of these theorems illustrate the theme of lifting functions from varieties to whole domains which began with Cartan's theorems A and B \cite{car51} and continued in \cite{agmc_vn} with bounds and later more precisely by \cite{agmc_dv, jkmc}, etc.

To give a concrete example of the noncommutative lifting principle, consider a Type I function given by the formula 
	$$h(z) = \ip{(A- z_1 Y - z_2 (1-Y) )\inv\alpha}{\alpha}$$
as in Theorem \ref{oldnevrepstype1}. We can extend $h$ to a function $H$ defined on $\Pi^2.$ Given a $Z \in \Pi^2_n$ we obtain
$H(Z)$ via the formula
	\beq \label{introtype1} H(Z) = (\alpha \otimes I_n)\ad (A \otimes I_n - Y \otimes Z_1 - (1-Y) \otimes Z_2)\inv (\alpha \otimes I_n). \eeq 
This function is well-defined on $\Pi^2$ and 
	$$H|_{\Pi_1^2} = h.$$
We show that free Pick functions $H: \Pi^d \to \cc\Pi^1$ that restrict to Type I Pick functions on $\Pi_1^d$ have representations analogous to that in \eqref{introtype1}.

\subsection{Free matrix monotonicity}

L\"owner's theorem on matrix monotone functions has been generalized to several commuting functional calculi.
Agler, {\mc}Carthy and Young proved that matrix monotone functions defined on commuting tuples of matrices extend
to functions in the L\"owner class \cite{amyloew}. Others have studied matrix monotonicity on an alternate commutative functional calculus involving tensor products \cite{han03,kor61,siva02}.

We prove a free analogue of L\"owner's theorem in several noncommuting variables.
We consider functions on domains contained in the \emph{real matrix universe,}
	$$\RR^d := \{X \in \M^d |  X_i = X_i^*, i =1 \ldots d\}.$$
A free set $D\subset \RR^d$ is a \emph{real free domain} if $D \cap \RR_n^d$ is open in the space $\RR^d_n.$ 
We prove L\"owner's theorem for real analytic free functions on convex free sets $D$, free sets such that $D_n = D\cap \M_n^d$ is convex for all $n$.  Since a free function 
is analytic if it is locally bounded \cite{vvw12}, a real free function $f$ is real analytic on $D$ if for all $X_0 \in D$ there is a bounded free function defined on a domain containing the set $\{X \in \M_{nk}^{d} | \norm{X - X_0 \otimes I_k} < \varepsilon\}$ which agrees with $f$ on $D$. We give a formal, intrinsically real definition of real analytic free functions in terms of powers series in Section \ref{realfreemaps}.

A real analytic free function is \emph{matrix monotone} if
	$$X \leq Y \Rightarrow f(X) \leq f(Y).$$
The following is our generalization of L\"owner's Theorem. 	
\begin{theorem}\label{freelow}
Let $D$ be a convex free set of $d$-tuples of self-adjoint matrices.
A real analytic free function $f: D \rightarrow \RR$ is matrix monotone if and only if
$f$ analytically continues to $\Pi^d$ as a function in the free Pick class.
\end{theorem}
This follows from Theorem \ref{lowTheorem} applied to functions on convex domains.

Convex free sets have a rigid structure because they are simultaneously free sets and convex sets at each level.  Convex free sets have been an object of recent work following the trend in semidefinite programming. For example, Helton and McCullough \cite{heltmc12} showed that semialgebraic convex sets are LMI domains, a generalization of Lasserre's result in the commutative case \cite{las01}. An LMI domain is the set of tuples of self-adjoint  $X_i$ such that
$
\sum A_i \otimes X_i \leq A_0
$
where each $A_i$ is symmetric.

We remark that the representations which correspond to our free analogue of Theorem \ref{oldnevrepstype1} can be used to manufacture  concrete formulas for monotone functions similar
to the representations of rational matrix convex functions in work of Helton, McCullough and Vinnikov
\cite{helmcvin06}, which show that rational matrix convex functions are all obtained
by taking the Schur complement of a monic LMI.
In the classical theory, there is a strong connection 
between convex and monotone functions \cite{bha97,kraus36}, which seems to be suggested again here.

We now introduce the machinery for understanding \emph{free power series.}
Let $\mathcal I$ denote the set of words in the letters $x_1, \ldots, x_d.$ The set $\mathcal I$ is
equipped with an involution $*$ which reverses the letters in a word. For example,
	$$(x_1x_2)^* = x_2x_1.$$
For a word $w \in \mathcal I,$ we define $X^w$ recursively. For the empty word $e$, $X^e = I$
and $X^{x_k w} = X_kX^w.$ For example,
	$(X_1,X_2)^{x_1x_2x_1} = X_1X_2X_1.$
A \emph{free power series} is an expression of the form
	$$f(X) = \sum_{I\in \mathcal I} c_IX^I.$$
To prove L\"owner's theorem, we prove an analogue of Corollary  \ref{seriesCharacter} for free power series. 
\begin{theorem}
Let  $f(X) = \sum_{I\in \mathcal I} c_IX^I$ be a free power series in $d$ noncommuting variables which converges for all $\|X\|\leq d + \epsilon.$
The function $f$ is matrix monotone on the set of $X$ such that $\|X\| \leq \frac{1}{d}$ if and only if for $1\leq k\leq d$, the infinite matrices
			$[c_{I^*x_kJ}]_{I,J \in \mathcal I}$
are positive semidefinite.
\end{theorem}
The preceding is proved as Theorem \ref{heltonHankelX}.

\subsubsection{L\"owner's theorem in one variable}
We now sketch the proof (in one variable) of Corollary \ref{seriesCharacter} to describe ideas important to the proof in several variables.

For real free analytic functions on convex sets, matrix monotonicity is equivalent to local matrix monotonicity.
Let $Df(X)[H]$ be the G\^ateaux derivative at a matrix $X$ in the direction $H.$
A free function is \emph{locally matrix monotone} when 
	\[
		Df(X)[H] \geq 0
	\]
for all $X \in \RR^d$, whenever $H\in \RR^d$ such that $H\geq 0.$
Our proof of Theorem \ref{freelow} constructs a formula for the derivative of a locally monotone real analytic free function that can be used to construct an analytic continuation of the function. We call these formulas \emph{models}.

Given the power series of a monotone function, we construct an explicit model for that function as follows. 
Consider an analytic function defined by a power series in one variable
\[
f(x) = \sum_{i=0}^\infty a_i x^i
\]
defined on a neighborhood of the closed disk.
We seek to study the formula,
\beq
f'(x) =
\bpm 1 \\ x \\ x^2 \\ \vdots \epm^* \bbm a_1 & a_2 & a_3 & \ldots \\ a_2 & a_3 & \phantom\ldots & \ldots \\ a_3 & \phantom a &  & \phantom a  \\ \vdots & \vdots & \phantom a & \ddots \ebm \bpm 1 \\ x \\ x^2 \\ \vdots \epm,
\eeq
in the free case. The derivative of $f$ evaluated via the functional calculus defined by \eqref{BFC} at a matrix $X$ in the direction $H$ is given by the formula,
\beq \label{hampos2}
Df(X)[H] =
\bpm 1 \\ X \\ X^2 \\ \vdots \epm^* \bbm a_1H & a_2H & a_3H & \ldots \\ a_2H & a_3H & \phantom\ldots & \ldots \\ a_3H & \phantom a &  & \phantom a  \\ \vdots & \vdots & \phantom a &  \ddots \ebm \bpm 1 \\ X \\ X^2 \\ \vdots \epm.
\eeq
We establish that
\beq \label{hampos}
\bbm a_1 & a_2 & a_3 & \ldots \\ a_2 & a_3 & \phantom\ldots & \ldots \\ a_3 & \phantom a &  & \phantom a  \\ \vdots & \vdots & \phantom a &  \ddots \ebm \geq 0.
\eeq
When \eqref{hampos} is satisfied, the formula \eqref{hampos2} for $f'$
implies that $f$ can be analytically continued to the upper half plane as a Pick function via the theory of models \cite{ampi}.

In one variable, a \emph{Hamburger model} for a function defined on a neighborhood
of the closed disk is an expression for the derivative of the form
\beq \label{hambonta}
Df(X)[H] =
\bpm 1 \\ X \\ X^2 \\ \vdots \epm^* \bbm a_{11}H & a_{12}H & a_{13}H & \ldots \\ a_{21}H & a_{22}H & \phantom\ldots & \ldots \\ a_{31}H & \phantom a &  & \phantom a  \\ \vdots & \vdots & \phantom a &  \ddots \ebm \bpm 1 \\ X \\ X^2 \\ \vdots \epm.
\eeq
where $(a_{ij})_{ij}$ is positive semidefinite. Note that if \eqref{hampos} is satisfied, then \eqref{hampos2} defines a Hamburger model for the function $f$.
We construct a Hamburger model in Section \ref{construction}, which is unique among expressions of the
form \eqref{hambonta}, irrespective of whether $(a_{ij})_{ij}$ is positive semidefinite or not.
In Section \ref{hankel}, we contruct a candidate for the Hamburger model from the power series for a function, which 
must be equal to the model obtained in Section \ref{construction} by uniqueness.

To construct the Hamburger model, we establish a sum of squares representation
\beq\label{sosDer}
	Df(X)[H] = \sum_{k=0}^{\infty} f_k\ad H f_k,
\eeq
and then reduce it algebraically to the model itself.
This method is modeled on the theory of Positivstellensatz\"e.
In \cite{helmc04}, Helton and McCullough showed the following Positivstellensatz, a generalization of a results of Schm\"udgen \cite{schmu91} and Putinar \cite{put93} to the free case.
\begin{theorem}[Helton, McCullough \cite{helmc04}]
Let $Q$ be a family of free polynomials with the technical assumption that if $q \in Q$ satisfies $q(X)\geq 0$ then $\norm{X} \leq M$ for some uniform constant $M$.
If $f(X)\geq 0$ whenever all $q\in Q$ satisfy $q(X)\geq 0,$ then
$$f = \sum g^*_ig_i + \sum h^*_jq_jh_j$$
for some finite sequences of free polynomials $g_i,$ $h_j,$ and $q_j$ where each $q_j\in Q.$
\end{theorem}

%We emulate the commutative L\"owner theorem \cite{amyloew} by working locally, and
%obtain a sum of squares representation as in equation \eqref{sosDer}. However, we differ in that our method is constructive. Typically, the generation of sum of squares representations or models follows from the Hahn-Banach theorem via a convexity argument.  In fact, free versions of the Hahn-Banach theorem exist, such as the Effros-Winkler theorem \cite{effwink}, for which generalizations of the classical argument for the existence of models have been executed \cite{heltmc13}.

To establish Equation \eqref{sosDer}, we apply the Choi-Kraus representation theorem \cite{bha07}.
A map $L: \M_n \rightarrow \M_n$ is called \emph{positive} if
	$H\geq 0 \Rightarrow L(H) \geq 0.$
A map $L: \M_n \rightarrow \M_n$ is said to be \emph{completely positive} if the extension of
$L$ to $L^m: \M_n \otimes \M_m \rightarrow \M_n \otimes \M_m$ given on simple tensors by the formula
	\beq  \label{composten} L^m(A \otimes B) = L(A) \otimes B \eeq
is positive for every $m.$
\begin{theorem}[Choi\cite{cho72}, Kraus \cite{kraus71}] \label{choi}
A completely positive linear map $L: \M_n \rightarrow \M_n$ can be written in the form
\[
	L(H) = \sum V_i\ad H V_i. 
\]
\end{theorem}

In one variable, since $Df(X)[H]$ is completely positive in $H,$ Theorem \ref{choi} is used to derive the equation \eqref{sosDer} locally, by establishing the $V_i$ are derived from free polynomials $u_i(X).$
So for each $X$,
\beq\label{sosDer2}
	Df(X)[H] = \sum_k u_k(X)\ad H u_k(X).
\eeq
The theory of establishing equation \eqref{sosDer2} for locally monotone functions corresponds to Lemma \ref{finitehamburger}.

To obtain a global version of \eqref{sosDer2}, we introduce the \emph{free coefficient Hardy space} $H^2_d$, the set of free power series in $d$ variables with coefficients in $l^2.$
The free coefficient Hardy space generalizes the classical Hardy space $H^2,$ the set of functions on the disk
whose power series coefficients are in $l^2.$ The classical Hardy space has many important properties, however most important to us
are the Szeg\"o kernels, $k_x \in H^2$ such that $\ip{f}{k_x}= f(x).$
We establish the theory of Szeg\"o kernels for $H^2_d.$ For a given $X \in \M^d_n$ we obtain $k^{ij}_X \in H_d^2$
such that $\ip{f}{k^{ij}_X}= f(X)_{ij}$ for $i,j \leq n,$ the \emph{Szeg\"o kernels at $X$.}
The general theory of the free coefficient Hardy space is given in Section \ref{fchardy}. 
The theory of free coefficient Hardy space codifies previously studied geometry of a vector of monomials, which was
used in the proof that positive free polynomials are sums of squares \cite[Section 2]{heltonPositive} and the proof of the noncommutative Schwarz lemma, as the \emph{noncommutative Fock space}\cite{hkms09}.
For example, in one variable, define $m_X$ to be the list of monomials written as a column vector. That is,
	$$m_X = \bpm 1 \\ X \\ X^2 \\ \vdots \epm.$$
For any $f\in H^2_1,$ there is a $u\in l^2$ such that $$f(X) = u^*m_X.$$
Via this duality, the theory of $m_X$ is, for our purposes, the theory of the free coefficient Hardy space. For this reason, we formally adopt the view that
$$H^2_1 = l^2(\{0,1,2,\ldots \})$$
so that for any $f = (c_i)^\infty_{i=0} \in H^2_1,$ and $X$ such that $\|X\|<1,$ the evaluation of $f$ at $X$ is defined via the formula
	$$f(X) = \sum^{\infty}_{i=0} c_i X^i.$$
The free coefficient Hardy space has appeared in operator theory as the \emph{noncommutative Hardy space}, where Popescu established the theory of composition operators \cite{popescuhardy}.

Note that each $u_k(X)$ from formula \eqref{sosDer2} can be written in the form of a tensored inner product,
\beq
u_k(X) = (\flattensor{v_k^*}{I})
\bpm 1 \\ X \\ X^2 \\ \vdots \epm
= \bpm v_{k1}I & v_{k2}I & v_{k3}I & \ldots \epm
\bpm 1 \\ X \\ X^2 \\ \vdots \epm
\eeq
where $v_k \in H_1^2$ is the vector of coefficients the polynomial of $u_k.$
The decomposition allows equation \eqref{sosDer2} to be written in the form of equation \eqref{hampos2},
\beq \label{hambonta2}
Df(X)[H] =
\bpm 1 \\ X \\ X^2 \\ \vdots \epm^* \bbm a_{11}H & a_{12}H & a_{13}H & \ldots \\ a_{21}H & a_{22}H & \phantom\ldots & \ldots \\ a_{31}H & \phantom a &  & \phantom a  \\ \vdots & \vdots & \phantom a &  \ddots \ebm \bpm 1 \\ X \\ X^2 \\ \vdots \epm,
\eeq
where the matrix $(a_{ij})_{ij}$ is positive semidefinite since $(a_{ij})_{ij} = \sum v_k v_k\ad.$

The matrix $(a_{ij})_{ij}$ in formula \eqref{hambonta2} is unique when restricted the space
spanned by the Szeg\"o kernels at $X.$
As the spectrum of $X$ grows, the Szeg\"o kernels exhaust the space and we obtain one matrix
$(a_{ij})_{ij}$ that satisfies equation \eqref{hambonta} for all $X$ that must agree with equation \eqref{hampos2} by uniqueness.
Since $(a_{ij})_{ij}$ is positive, equation \eqref{hampos} is satisfied, which allows us to conclude that $f$ extends to the upper half plane as a Pick function.
The full construction is given in Section \ref{construction}.

\subsection{The structure of the paper}

The paper is structured as follows. In Section \ref{background} we describe free analysis in detail and how it relates to the classical
functional calculus. In Section \ref{foundations}, we develop the foundations of our paper. The begins with a discussion of
the free coefficient Hardy space. Then, we establish the \emph{lurking isometry argument,} a tool
to represent functions.
In Section \ref{LOWNER}, we prove L\"owner's theorem, using the method described above.
In Section \ref{NevSect}, we prove our Nevanlinna representations, and characterize them using asymptotic behavior.

\section{Background} \label{background}
We fix $\hilbert$ to be a infinite-dimensional separable Hilbert space.

\subsection{Free analysis}

\subsubsection{Matrix universe}
	Let $\M_n(\mathbb{C})$ denote the $n \times n$ matrices over $\mathbb{C}.$
	We define the \emph{matrix universe} to be
		$$\M = \bigcup \M_n(\mathbb{C}),$$
	and	the \emph{$d$-dimensional matrix universe} by
		$$\M^d = \bigcup \M_n(\mathbb{C})^d.$$
	Most function theory is proven for maps from $\M^d$ to $\M^{d'}.$ However, we require a slight generalization of this calculus.
	Let $\vecspace$ be a vector space over $\mathbb{C}.$
	We define the \emph{$\vecspace$ matrix universe} as 
		$$\tensor{\vecspace}{\M} = 	\bigcup \tensor{\vecspace}{\M_n(\mathbb{C})}.$$
	%Let $\mathcal{H}$ and $\mathcal{K}$ be Hilbert spaces.
	%We define the \emph{$\BHK$ matrix universe}
	%	$$\tensor{\BHK}{\M} = 	\bigcup \tensor{\BHK}{\M_n(\mathbb{C})}.$$
	%Among these is the $d$-dimensional matrix universe when viewed as
	%	$$\M^d = \tensor{\B(\mathbb{C},\mathbb{C}^d)}{\M} =
	%	\tensor{\mathbb{C}^d}{\M} =  \bigcup \M_n(\mathbb{C})^d.$$
	%Similarly, for a Hilbert space $\hilbert$ we will consider the \emph{$\hilbert$ matrix universe.}
	%	$$\tensor{\hilbert}{\M} = \tensor{\B(\mathbb{C},\hilbert)}{\M}.$$
		
	Similarly, we define the \emph{real matrix universe} to be
		$$\RR = \{X\in \M | X = X\ad\},$$
and	the \emph{real $d$-dimensional matrix universe} by
		$$\RR^d = \bigcup \RR_n^d.$$
	Let $\vecspace$ be a vector space over $\R.$
	We define the \emph{real $\vecspace$ matrix universe}
		$$\tensor{\mathcal V}{\RR} = 	\bigcup \tensor{\vecspace}{\RR_n}.$$
	
	The \emph{vertical tensor} decomposition of data reflects the structure of our arguments.
	We take the opinion that the first slot contains \emph{extrinsic data,}
	data related to the evaluation of functions, and that the second slot contains
	\emph{intrinsic data,} information about the input and output of functions.
	Data that has been decomposed into its intrinsic and extrinsic parts will be presented visually as
		$$\bigtensor{\text{Extrinsic data}}{\text{Intrinsic data}}$$
	for organizational purposes. We have denoted our tensor products vertically to make formulas more clear.
	For example, we desire
		$$\tensor{A}{B}\tensor{C}{D} = \tensor{AC}{BD}.$$
	Moreover, we will often encounter expressions of the form of Equation \ref{hampos2}, which can be rewritten in vertical tensor notation as
	$$\bpm 1 \\ X \\ X^2 \\ \vdots \epm^* \tensor{\bbm a_1 & a_2 & a_3 & \ldots \\ a_2 & a_3 & \phantom\ldots & \ldots \\ a_3 & \phantom a &  & \phantom a  \\ \vdots & \vdots & \phantom a &  \ddots \ebm}{H} \bpm 1 \\ X \\ X^2 \\ \vdots \epm$$
	which is symmetric. We believe adding the visual symmetry induced by adopting vertical tensor notation will make our arguments more clear.

	The matrix universe $\M$ has a direct sum, and this is inherited by the tensor product as follows on
	simple tensors
		$$ \tensor{A}{B} \oplus \tensor{C}{D} = \tensor{A}{B\oplus 0} + \tensor{C}{0 \oplus D}$$
	and is extended by linearity on the entire tensor product.
	
	We will sometimes use \emph{flat tensors} of two pieces of intrinsic or extrinsic data.
	The flat tensor $\otimes$ is implemented so that it is right distributive over $\oplus,$
			$$A \otimes (B \oplus C) = (A \otimes B) \oplus (A \otimes C)).$$

\subsubsection{Free maps}
	A \emph{free set} $D \subseteq \tensor{\vecspace}{\M}$ satisfies the following axioms
		\begin{enumerate}
			\item $A, B\in D \Leftrightarrow A \oplus B \in D,$
			\item $\forall A \in D\cap \tensor{\mathcal V}{\M_n(\mathbb{C})},
			\forall U\in \textrm{U}_n(\mathbb{C}),
			\tensor{I}{U^*}A \tensor{I}{U}\in D.$
		\end{enumerate}
	
	A \emph{free domain} $D \subset \M^d$ is a free set in the $d$-dimensional matrix universe that is open in the disjoint union topology.
	Given two free sets $D$ and $D',$ a \emph{free map}
	$f: D \rightarrow D'$ is a map so that the graph of $f$ is a free set and
	if for some invertible matrix $S$, $A \in D,$ and $\tensor{I}{S^{-1}}A \tensor{I}{S}\in D,$ then
	$$f(\tensor{I}{S^{-1}}A \tensor{I}{S}) = \tensor{I}{S^{-1}}f(A ) \tensor{I}{S}.$$
\subsubsection{Real free maps}\label{realfreemaps}
	A \emph{real free set} $D \subseteq \tensor{\vecspace}{\RR}$ satisfies the following axioms
	\begin{enumerate}
		%\item %$A\in D \Rightarrow A = A^*$
		\item $A, B\in D \Leftrightarrow A \oplus B \in D$
		\item $A \in D\cap \tensor{\vecspace}{\RR_n},
		U\in \textrm{U}_n(\mathbb{C})  \Rightarrow \tensor{I}{U^*}A \tensor{I}{U}\in D$
	\end{enumerate}	
	
	A \emph{real free domain} is a real free set  in the $d$-dimensional matrix universe, $D \subset \RR^d,$
	that is open in the disjoint union topology
	restricted to self-adjoint matrices. A \emph{real free map} $f: D \rightarrow D'$
	is a map so that the graph of $f$ is a real free set and if for some unitary $U,$
	$A \in D,$ and $\tensor{I}{U^*}A \tensor{I}{U}\in D,$ then
	$$f(\tensor{I}{U^*}A \tensor{I}{U}) = \tensor{I}{U^*}f(A ) \tensor{I}{U}.$$
	
	We define the \emph{complexified tangent bundle} of a domain as follows, noting that we need the vector component of the tangent bundle to have the same dimension as the point over which it lies. Let $D \subset \RR^d$ be a real free domain (resp. free domain). The complexified tangent bundle (resp. tangent bundle) is the set $T(D)$ given by the formula
		$$T(D) = \bigcup_n D_n \times \M_n^d,$$
	where $D_n = D\cap \M_n^d$.
	
	We now begin the discussion of power series and real analyticity. We adopt the convention that if $w$ is a word in the letters $x_1, \ldots x_k,$ and
	$X = (X_1, \ldots X_k),$ then $X^w = w(X).$ For example $X^{x_1x_2x_1} = X_1X_2X_1.$
	Furthermore, we define an involution $*$ on words which reverses their letters. For example
	$(x_1x_2)^* = x_2x_1.$ We use $|w|$ to denote the length of a word.

	To discuss analyticity, we first define a meaningful way to talk about
	local coordinates.
	\begin{definition}
		A \emph{$n$-frame} is a basis $F = (F_i)^{dn^2}_{i=1}$ of $\RR^d_n$ such that
		$(F_i)_j$ is positive semidefinite over all $i, j$, where $(F_i)_j$ is the matrix in the $j$th slot of the $i$th basis element.
		
		We denote the \emph{local coordinate function} with respect to $F$ as
			$$X_F =  \sum^{dn^2}_{i=1} \flattensor{F_{i}}{X_{i}}.$$
	\end{definition}
	
	A real free map is \emph{analytic} if for each point $A_0 \in D \cap \RR_n^d$,
	the function of $dn^2$ noncommuting variables has a power series for each of its noncommuting entries.
	\begin{definition}
		Let $f: D \rightarrow \RR$ be a real free function.
		We say $f$ is \emph{real analytic} if for any $A_0 \in D \cap \RR_n^d$, $n$-frame $F$,
		and vectors $u, v\in \mathbb{C}^n$,
		there are $c_I \in \C$ such that
			$$\flattensor{u^*}{I} f(\flattensor{A_0}{I} + X_F) \flattensor{v}{I} = \sum_I c_IX^I$$
		for all $X \in \RR^{dn^2}$ such that $\|X\|<\epsilon.$
	\end{definition}
	
	We note that the definition of a real free analytic map $f$ implies that for each $A_0\in D,$ there is
	a free domain $D',$ the domain of convergence of the power series for $f$ at $A_0$,
	containing $A_0$ such that f analytically continues to $D'\cap D$ as a function
	of the entries of the input as tuples of matrices. However, \emph{a priori},
	it is not clear that the continuation of $f$ is a free function. 
	This issue is resolved by the following lemma combined with observation that the derivative is a linear map.
	\begin{lemma}\label{commutator}
		Let $D \subset \M_n^d.$ Let $f: D \rightarrow \M_n$ be a differentiable function.
		Let $[\cdot,\cdot]$ denote the commutator. (That is, $[X,Y]=XY-YX.$)
		\begin{enumerate}
			\item $f$ respects unitary similarity if and only if for all $A \in \RR_n$
				$$Df(X)[[iA,X]] = [iA,f(X)].$$
			\item f respects similarity if and only if for all $T\in \M_n$
				$$Df(X)[[T,X]] = [T,f(X)].$$
		\end{enumerate}
	\end{lemma}
	\begin{proof}
		Suppose $f$ respects unitary similarity.
		Let $A \in \RR_n$
		Since $f$ respects unitary similarity,
			$$e^{-itA}f(X)e^{itA}= f(e^{-itA}Xe^{itA}).$$
		Differentiating this equation at $0$ via the chain rule gives
			$$Df(X)[[iA,X]] = [iA,f(X)].$$
			
		The fundamental theorem of calculus proves the converse.
		That is,
			since
				$$\frac{d}{dt} f(e^{-itA}Xe^{itA}) = Df(e^{-itA}Xe^{itA})[[iA,e^{-itA}Xe^{itA}]]$$
			and
				$$\frac{d}{dt} e^{-itA}f(X)e^{itA} = [iA,e^{-itA}f(X)e^{itA}],$$
			then
				$$Df(e^{-itA}Xe^{itA})[[iA,e^{-itA}Xe^{itA}]] = [iA,e^{-itA}f(X)e^{itA}]$$
		implies $$f(e^{-itA}Xe^{-itA}) = e^{-itA}f(X)e^{itA} + C$$ and $C$ must be $0$ since evaluating at $0$
		gives $f(X) = f(X) + C.$ Since, every unitary is of the form $e^{iA}$ for some self-adjoint $A$,
		(see \cite[Chapter 2]{helgason})
		we are done.
			
		The proof of $2$ is similar and is left to the reader.
	\end{proof}
	Since the derivative
	is complex linear, Lemma \ref{commutator} immediately implies the following. 
	\begin{corollary}\label{commieCor}
		 A complex analytic real free function is a free function. 
	\end{corollary}

	\subsubsection{The real free derivative identity}
	
	In the work of Helton, Klep, and McCullough \cite{helkm11},
	Voiculescu \cite{voi04,voi10} and Kaliuzhnyi-Verbovetskyi and Vinnikov \cite{vvw12}, a number of identities have
	been proven to compute derivatives of free functions.
	We will need the following two identities for real free functions. The following identity is given for commutative matrix functions in the proof of the commuting multivariable L\"owner theorem of Agler, {\mc}Carthy and Young \cite{amyloew} and proven in \cite{bickel13}.
	\begin{proposition}\label{derUnit}
	Let $f: D \rightarrow \M$ be a differentiable real free function.
		$$Df\bpm
			X & \\
			 & Y	
		\epm \bbm
			 & X-Y \\
			 X-Y & 	
		\ebm =
		\bpm
			 & f(X)-f(Y)\\
			 f(X)- f(Y) & 	
		\epm$$
	\end{proposition}
	\begin{proof}
		Let
		$A =
		\bpm
			0 & -i \\
			i & 0
		\epm.$
		By Lemma \ref{commutator}
		$$Df\bpm
			X & 0 \\
			0 & Y
			\epm \left[
			\left[iA,\bpm
			X & 0 \\
			0 & Y
			\epm\right]\right]
			=
			\left[iA,\bpm
			f(X) & 0 \\
			0 & f(Y)
			\epm\right].
		$$
		Substituting $A$ and simplifying obtains the desired result.
		$$Df\bpm
			X & \\
			 & Y	
		\epm\bbm
			 & X-Y \\
			 X-Y & 	
		\ebm =
		\bpm
			 & f(X)-f(Y)\\
			 f(X)- f(Y) & 	
		\epm.$$
		
	\end{proof}

\subsubsection{Dominating points}
	Given a finite set of ordered pairs
		$$(x_1,y_1), \ldots (x_n,y_n) \in \R^2,$$
	such that all $x_i$ are distinct,
	there is a real analytic function $f:\mathbb{R} \rightarrow \mathbb{R}$
	such that $f(x_i) = y_i.$
	
	In the free setting this assertion is no longer true.
	For example, the value of a function $f$ at
		$$X = \bpm
			1 &  &\\
			  & 2 & \\
			    & & 3
		\epm$$
	completely determines the value at the point
		$$Y = \bpm
			1 &  \\
			  & 2 
		\epm$$
	since
	$$f\bpm
			1 &  &\\
			  & 2 & \\
			    & & 3
	\epm
	=
	\bpm
			f\bpm 1 &  \\
			  & 2 \epm & \\ 
			     & f(3)
	\epm.$$

	We now draw on the notion of dominating points which was used in Helton and McCullough \cite{heltmc12}.
	\begin{definition}
		Let $X, Y \in \RR^d.$ We say \emph{$X$ dominates $Y$}  if
			$f(X) = 0$ implies that $f(Y)= 0$.
	\end{definition}

	Given an $X,$ we desire to find an algebraically well-conditioned point which dominates $X.$ 
	\begin{definition}
		Let $X \in \RR^d.$ We say that $X$ is \emph{reduced} if
		the algebra generated by $X_1, \ldots X_d$ is equal to 
		 	 $$\bigoplus_i M_{n_i}(\mathbb{C})$$
		for some finite sequence of integers $n_i.$
	\end{definition}	
	
	Elementary techniques from the theory of finite dimensional C$\ad$-algebras provide a reduced dominating point for a given $X.$
	\begin{lemma} \label{domReduction}
		Let $X \in \RR^d.$ There exists an $X^0 \in \RR^d$ such that
		$X^0$ dominates $X,$ $X$ dominates $X^0$, and
		$X^0$ is reduced.
	\end{lemma}
	\begin{proof}
	Let $\mathcal{A}$ denote the algebra generated by the components of $X.$
	Note that $\mathcal{A}$ is a $*$-algebra since the generators are self-adjoint.
	
	By an Artin-Wedderburn type theorem for finite dimensional $C^*$-algebras \cite[Theorem III.1.1]{da06},
	$\mathcal{A}\cong \bigoplus_i M_{n_i}(\mathbb{C})$ for some finite sequence of integers $n_i.$
	So let $\pi:\bigoplus_i M_{n_i}(\mathbb{C}) \rightarrow \mathcal{A}$ be a $*$-isomorphism.
	By a structure theorem in \cite[Theorem III.1.2]{da06},
	there is a unitary $U$ and integers $m_i$ such that for
	every $ \bigoplus_i Y_i\in\bigoplus_i M_{n_i}(\mathbb{C}),$ we can write
	the homomorphism via the formula
	$\pi(\bigoplus_i Y_i) = U^* \bigoplus_i (Y_i \otimes I_{m_i}) U.$
	
	Let $X^0 = \pi^{-1}(X).$ Since real free functions respect direct sums and
	unitary similarity, $X^0$ dominates $X$ and $X$ dominates $X^0.$
	Since the coordinates of $X^0$ generate $\bigoplus_i M_{n_i}(\mathbb{C}),$ $X^0$ is reduced.
	\end{proof}
	
	We now will show that if $X$ dominates $Y$ then the derivative of a function $f$ at $X$ determines
	the derivative of $f$ at $Y.$
	\begin{lemma}\label{dominationDer}
		Let $X \in \RR_n^d, Y \in \RR_m^d$ be points such that $X$ dominates $Y.$
		Let $L: \RR_n^d \rightarrow \RR_n$ be a linear map.
		There is a unique map $K: \RR_m^d \rightarrow \RR_m$ such that if
		$f$ is a real analytic free function,
		and $Df(X) = L$ as a linear map, then
		$Df(Y)= K.$
	\end{lemma}
	The proof of the lemma follows immediately from the following free identity.
	\begin{proposition} \label{tenDer}
		Let $f: D \rightarrow \M$ be a differentiable real free function.
		Let $A \in \RR^d$ and $B\in \RR$.
		$$Df(X \otimes I_n)[A \otimes B] =
		Df(X)[A]\otimes B.$$
	\end{proposition}
	\begin{proof}
		Let $U\in U_n$ be unitary such that
			$$B = U^* D U$$
		where $D$ is some real diagonal matrix.
		
		So,
		\begin{align*}
			Df(X \otimes I_n)[A \otimes B] & = 
			(I\otimes U^*) Df(X \otimes I_n)[A \otimes D] (I\otimes U)
			\\ & = 
			(I\otimes U^*) Df(\bigoplus^n_{i=1} X)[\bigoplus^n_{i=1} d_iA] (I\otimes U)
			\\ & = 
			(I\otimes U^*) \bigoplus^n_{i=1}  Df(X)[d_iA] (I\otimes U)
			\\ & = 
			(I\otimes U^*) \bigoplus^n_{i=1} d_iDf(X)[A] (I\otimes U)
			\\ & = 
			(I\otimes U^*) Df(X)[A] \otimes D (I\otimes U)
			\\ & = 
			Df(X)[A]\otimes B.
		\end{align*}	
	\end{proof}
	\begin{proof}[Proof of Lemma \ref{dominationDer}]
		Let $X^0$ be the reduction of $X$ given in the proof of Theorem \ref{domReduction}.
		By Proposition \ref{tenDer}, the derivative at $X^0$ is determined by $L$,
		since there is a $k,$ a unitary $U$ and a matrix tuple $W$
		such that $X_0 \otimes I_k = U\ad X \oplus W U.$
		That is, there is an $L_0$
		such that for any $f$ such that $Df(X) = L,$ $Df(X^0) = L_0.$
		It can be shown that there is a homomorphism $\pi$ taking the algebra generated by the components of $X^0$
		to the algebra generated by the components of $Y$ so that $\pi(X^0_i)= Y_i.$ Furthermore, for
		some unitary $V,$
			$$\pi(\bigoplus_l Z_l) = V^* \bigoplus_l (Z_l \otimes I_{m_i}) V.$$
		Thus, by Proposition \ref{tenDer},
		the derivative at $X^0$ determines the derivative at $Y.$
	\end{proof}

\section{Foundations} \label{foundations}
\subsection{The free coefficient Hardy space}\label{fchardy}
	 Introduced in \cite{har15}, and \cite{har-lit},
	 the classical Hardy space $H^2$ is typically defined as the set of analytic functions
	on the disk so that
		$$\lim_{r\rightarrow 1} \frac{1}{2\pi}\int^{\pi}_{-\pi} |f(re^{i\theta})|\,\dd\theta \leq \infty.$$
	The space $H^2$ is endowed with an inner product given by the formula
		$$\ip{f}{g} = \lim_{r\rightarrow 1} \frac{1}{2\pi}\int^{\pi}_{-\pi} f(re^{i\theta})\overline{g(re^{i\theta})}\,\dd\theta$$
	and is indeed a Hilbert space.
	In an alternative characterization, the functions $z^n$ form an orthonormal basis for $H^2.$ Thus, $H^2$ is
	also the set of functions on the disk such that their coefficients in a power series at the
	origin are sequences in $l^2.$
	
		In the classical Hardy space $H^2,$ there exists a function, the Szeg\"o kernel \cite{sze21b,kra}, $k_\alpha$
	such that if $f \in H^2,$ then
		$$f(\alpha)= \ip{f}{k_\alpha}.$$
	
		We generalize the second interpretation of the Hardy space as the \emph{free coefficient Hardy space}.
	\begin{definition}
		Let $\mathcal{I}$ be the set of monomial indices in the free algebra with $d$ variables.
		The \emph{free coefficient Hardy space} $H^2_d$ is $l^2(\mathcal{I})$ where for
		$f \in H^2_d$ such that $f = (a_I)_{I\in \mathcal{I}}$ the value of $f$ at $X$
		is defined on tuples of matrices $X$ of norm less than $\frac{1}{d}$ by the formula
			$$f(X) = \sum_{I\in \mathcal{I}} a_IX^I.$$
	\end{definition}
	
		Importantly, this space has a Szeg\"o kernel itself.
	\begin{definition}\label{szego}
		Let $\mathcal{I}$ be the set of monomial indices in the free algebra with $d$ variables.
		Let $X \in \M^n.$
		Define the monomial basis vector $m_X = (X^I)_{I\in \mathcal{I}}.$
		The Szeg\"o kernel is given by $k^{ij}_X = (\overline{(m^X_{I})_{ij}})_{I\in \mathcal{I}}.$ That is, it is the sequence
		$(i,j)$-th entries of each monomial $I$ evaluated at $X.$
	\end{definition}
		Note that by a straightforward calculation,
			$$f(X)_{ij} = \ip{f}{k^{ij}_X}.$$
		Thus, the behavior of $f$ at $X$ is determined by the projection of $f$
		onto the space spanned by $k^{ij}_X.$
	\begin{definition}\label{szegospace}
		Define 
			$$\vecspace_X = \operatorname{span}_{ij}\{k^{ij}_X\} \subset H^2_d.$$
		Furthermore, define $P_X: H^2_d \rightarrow \vecspace_X$ to be the projection onto $\vecspace_X.$
%		\red I don't see any motivation for giving these a name at this juncture, if other people think they're useful, they'll give them a name, but i dont thin it adds anything. \black 		
%		\blue I want to be more precise about this. For functions in the space $H^2_d$,  at a point $X$ define by $P_X$ the local projection onto $V_X$? \black  
	\end{definition}
		Thus, given some pair $(X, Y)$ such that there is a free polynomial with $p(X) = Y$,
		we can correspond to $p$ a unique $f$ of minimum norm in $H^2_d$ such that $f(X) = Y$ by assigning
		$f = P_X p.$
	
		The following is essentially the statement that the values of a function determine the function.
	\begin{proposition}\label{denseSpan}
		$$\overline{\bigcup_{\|X\|\leq \frac{1}{d}} \vecspace_X} = H^2_d.$$ 
	\end{proposition}
	\begin{proof}
		Suppose $\overline{\bigcup_{\|X\|\leq \frac{1}{d}} \vecspace_X} \neq H^2_d.$ Then there exists a nonzero function
		$f \in\overline{\bigcup_{\|X\|\leq \frac{1}{d}} \vecspace_X}^\perp.$
		So, $0 = \ip{f}{k^{ij}_X}$ for all $i, j$ and $X$ which implies that $f$ is zero. This is a contradiction.
	\end{proof}
		In general we will often forget coordinate systems, so we use an alternative characterization of $\vecspace_X.$
	\begin{proposition}\label{altszegspace}
		The space $\vecspace_X$ is the unique vector space such that
		$$\tensor{\vecspace_X}{\mathbb{C}^n} = \mathrm{span}\{\tensor{I}{U} m_Xc| c \in \mathbb{C}^n, U \in U_n\}.$$
	\end{proposition}
	To prove this we need a decomposition theorem, a \emph{detensoring lemma} that will allow us to show that spaces like the one
	in the above lemma are well-defined.
	\begin{lemma}[Detensoring lemma]\label{stsc}
		For a Hilbert space $\hilbert$, let $\vecspace$ be a subspace of $\tensor{\hilbert}{\mathbb{C}^n}$ such that if $U \in \mathcal U_n$ then
		$$\tensor{I}{U} \vecspace = \vecspace.$$ Then there exists $\vecspace'$, a subspace of $\hilbert$, so that $\vecspace = \tensor{\vecspace'}{\mathbb{C}^n}.$
	\end{lemma}
	\begin{proof}
		Let $e_1, \ldots, e_n$ be an orthonormal basis for $\mathbb{C}^n.$
		
		Let $\vecspace'$ be the subspace $\{v | \tensor{v}{e_1} \in \vecspace\}.$
		
		We first show $\vecspace \subseteq \vecspace'\otimes \mathbb{C}^n.$
		Let $v \in \vecspace.$
		Then $$v = \sum_i \tensor{v_i}{e_i}.$$
		Let $U_i$ be the unitary that fixes $e_i$ and sends $e_j$ to $-e_j$ if $i \neq j.$
		So, $$\tensor{I}{U_i}v  = \sum_i (-1)^{\chi(i\neq j)} \tensor{v_i}{e_i} \in \vecspace$$ by assumption. Therefore
		$\frac{1}{2}\left(\tensor{I}{U_i}v + v\right) = \tensor{v_i}{e_i} \in \vecspace.$ Let $W_i$ be a unitary taking $e_i$ to $e_1.$
		So $$\tensor{I}{W_i}\tensor{v_i}{e_i} = \tensor{v_i}{e_1}.$$
		Thus, each $v_i \in \vecspace'$ and so $v \in \tensor{\vecspace'}{\mathbb{C}^n}$.
		
		We now show $\vecspace \supseteq \tensor{\vecspace'}{\mathbb{C}^n}.$
		Suppose $v \in \tensor{\vecspace'}{ \mathbb{C}^n},$
		so that, $$v = \sum_i \tensor{v_i}{e_i}$$ where each $v_i \in \vecspace'.$ Let $W_i$ be a unitary taking $e_1$ to $e_i.$
		Thus, $$v = \sum_i \tensor{I}{W_i} \tensor{v_i}{e_1} \in \vecspace$$ since each $\tensor{I}{W_i}
		\tensor{v_i}{e_1} \in \vecspace$ by definition.
	\end{proof}
	
	We now prove Proposition \ref{altszegspace}.
	\begin{proof}[Proof of Proposition \ref{altszegspace}]
		By the detensoring lemma, there is a unique vector space $\vecspace$ such that
			$$\tensor{\vecspace}{\mathbb{C}^n} = \mathrm{span}\{\tensor{I}{U} m_Xc| c \in \mathbb{C}^n, U \in U_n\}.$$
		
		Let $U_1$ be the unitary sending $e_i$ to $-e_i$. 
		Note, 
		$$\tensor{k_X^{ij}}{e_i} = \frac{1}{2}\left[\tensor{I}{I}m_Xe_j - \tensor{I}{U_1}m_Xe_j\right].$$  Thus, $\vecspace$ contains $\vecspace_X$ as the $k_X^{ij}$ form a basis for $\vecspace_X.$
		
		Let $E_{ij}$ designate the matrix in $\M_n$ with $1$ in the $ij$th position and $0$ elsewhere. Note, $m_X = \sum \tensor{k_X^{ij}}{E_{ij}}.$
		So, since the unitary matrices in $\M_n$ span $\M_n$ itself, it can be derived that
			$$\mathrm{span}\{\tensor{I}{U} m_X W| U, W \in U_n\}= \mathrm{span}\{ \tensor{k_X^{ij}}{A} | A \in \M_n \}.$$
		Now, if $f \in \vecspace,$ there is some $l\in \mathbb{C}^n$ such that
			$$\tensor{f}{l} = \sum_{ij} \tensor{k_X^{ij}}{A}c_{ij}.$$
		Note that $Ac_{ij} = l,$ since $k_X^{ij}.$ So
			$\tensor{f}{l} = \sum_{ij} \tensor{k_X^{ij}}{l}.$
		So $\tensor{\vecspace}{\mathbb{C}^n} \subseteq \tensor{\vecspace_X}{\mathbb{C}^n}$ and thus $\vecspace_X$ contains $\vecspace$.
	\end{proof}

	\subsubsection{The coefficient Hardy space of a general free vector-valued function}
		The free coefficient Hardy space is useful for interpolation problems to obtain existence
		and uniqueness results. If we relax the choice of basis of functions, we still obtain
		uniqueness results. For our constructions, this is often enough.
			
		\begin{definition}
		Let $\ph$ be a free function on a domain $D \subset \M^d$ and taking values in $\tensor{\hilbert}{\M}$.
		The \emph{free coefficient Hardy space for $\ph$}, denoted by $H^2_\ph$ is given by 
		$$H^2_\ph  = \{u\in \hilbert|\forall X\in D, \tensor{u^*}{I}\ph(X) = 0\}^\perp.$$
		The value of $f\in H^2_\ph$ at $X\in D$ is defined to be:
			$$f(X) = \tensor{f^*}{I}\ph(X).$$
		\end{definition}
	
		We define the vector spaces from Definition \ref{szegospace} in the second abstract characterization which is easier to
		state in this context. However, a choice of basis will yield Szeg\"o kernels as in the original definition.
		\begin{definition} \label{localvec}
		Let $\ph$ be a free function on a domain $D \subset \M^d$ and taking values in $\tensor{\hilbert}{\M}$, and let $X \in \M^d_n.$
		The space $\vecspace^\ph_X \subset H^2_{\ph}$ is the unique vector space such that
		$$\tensor{\vecspace^\ph_X}{\mathbb{C}^n} = \mathrm{span}\{\tensor{I}{U} \ph(X)c| c \in \mathbb{C}^n, U \in U_n\}.$$
		
		Define the projection $P^\ph_X: H^2_\ph \rightarrow \vecspace^\ph_X$ to be the projection onto $\vecspace^\ph_X.$
		\end{definition}
		
		We note again that for $f,g \in H^2_\ph,$ $f(X) = g(X)$ if and only if $P^\ph_Xf=P^\ph_Xg.$
	Furthermore, the spaces $\vecspace^{\ph}_X$ exhaust $H^2_\ph.$
	\begin{proposition}\label{denseSpanGen}
		$$\overline{\bigcup_{X\in D} \vecspace^{\ph}_X} = H^2_\ph.$$ 
	\end{proposition}
	\begin{proof}
		Suppose $f \in (\bigcup_{X\in D} \vecspace^{\ph}_X)^{\perp}.$
		So, for every $X,$ $P_X f = 0 = P_X 0.$ Thus, $f(X) \equiv 0.$
	\end{proof}

\subsection{Models}

	As in the commutative case, model formulas are a powerful tool for investigating
	various classes of holomorphic functions on different domains. For example,
	for scalar valued functions in the Schur class in two variables, we have the following theorem:
	\begin{theorem}[Agler \cite{ag90}] \label{classicalschurmodel}
		Let $U \subseteq \D^2$.
		Let $\ph:U \to \D$.
		There is an analytic continuation of $\ph$ to $\D^2,$ $\ph:\D^2 \to \D$ if and only if
		there exists a separable Hilbert space $\hilbert,$
		an orthogonal decomposition $\hilbert = \hilbert_1 \oplus \hilbert_2$, and
		a holomorphic function $u:U \to \hilbert$ so that 
		\[
		1 - \cc{\ph(\mu)}\ph(\la) = \ip{(1 - \mu\ad\la)u_\la}{u_\mu},
		\]
		where $\la$ is viewed as the operator $\la = \la_1 P_{\hilbert_1} + \la_2 P_{\hilbert_2}$.
	\end{theorem}

 To use models for analytic continuation techniques in the free setting we use the following interpolation theorem.
 	Let $\B^d$ be the free set of $d$-tuples of contractions
		\beq\label{dcontractions}\B^d = \{X = (X_1, \ldots, X_d) \in \M^d | \norm{X_i} < 1 \text{ for } i = 1, \ldots, d\}.\eeq
	\begin{theorem}[Agler, {\mc}Carthy \cite{agmc_gh}] \label{ncschur3}
		Let $D \subset \B^d$ be a free set.
		Let $\ph: D \rightarrow \B$ be a free function.
		There is an extension of $\ph$ to $\B^d$ as a free function $\ph:\B^d \to \B$
		if and only if there are
		functions $u_1, \ldots, u_d: D \rightarrow \tensor{\B(\mathbb{C},\hilbert)}{\M}$  so that for any $X,Y \in D$, 
		\beq \label{ncschurmodel}
			I - \ph(Y)\ad\ph(X) = \sum_i u_i(Y)\ad \tensor{I}{I - Y_i\ad X_i} u_i(X).
		\eeq
	\end{theorem}
	
	Transformed to the upper half plane via a M\"obius transform, the model theorem is as follows.
	\begin{theorem} \label{ncschur2}
		Let $D \subset \Pi^d$ be a free set.
		Let $\ph: D \rightarrow \Pi$ be a free function.
		There is an extension of $\ph$ to $\Pi^d$ as a free function $\ph:\Pi^d \to \Pi$
		if and only if there are
		functions $u_1, \ldots, u_d: D \rightarrow \tensor{\B(\mathbb{C},\hilbert)}{\M}$  so that for any $X,Y \in D$, 
		\beq \label{ncpickmodel}
			 \ph(X) - \ph(Y)\ad = \sum_i u_i(Y)\ad \tensor{I}{X_i - Y_i\ad } u_i(X).
		\eeq
	\end{theorem}

\subsection{The lurking isometry argument for linear forms}

	The proof of Theorem \ref{ncschur3} relies on the existence of a free version of the standard lurking isometry argument. The additional complications inherent in the free algebraic structure of the models warrant an argument establishing the existence of these isometries in the free case. We prove a lurking isometry argument for linear forms, which makes slightly different assumptions, but can also be used to make model theoretic arguments.

	\begin{proposition}\label{localunitary}
	Suppose that for free functions $\theta$ and $\ph$ taking a set $D \subset \M$ into $\tensor{\hilbert}{\M}$, we have the relation
	\[
	\theta(X)\ad\tensor{I}{H}\theta(X) = \ph(X)\ad \tensor{I}{H}\ph(X)
	\]
	for all $(X, H) \in T(D)$.
	Let $\vecspace_X^\ph$ and $\vecspace_X^\theta$ be the vector spaces given in Definition \ref{localvec}. Then for each $X \in D\cap\M_n^d$, there exists a unique unitary operator $U_X: \vecspace^\theta_X \to \vecspace^\ph_X$ such that for any unitary $U \in U_n$ and any $c \in \C^n$, 
	\[
	\tensor{U_X}{I} \tensor{I}{U}\theta(X)c = \tensor{I}{U}\ph(X)c.
	\]
	\end{proposition}
	\begin{proof}
	%We will use the standard lurking isometry argument to establish this claim.
	%Suppose that $\theta$ and $\ph$ are as given. Suppose that $X \in \M^d_n$ is in $D$.
	%First, for $\theta$ and $\ph$, if $U$ is unitary, by the unitary similarity property of free functions,
	%\begin{align*}
	%\tensor{I}{U} \ph(X) & = \tensor{I}{U}\ph(X)U\ad U \\
	%&= \ph(UXU\ad)U.
	%\end{align*}
	Let $c_1, c_2 \in \C_n$ and $U_1, U_2 \in \mathcal U_n$. Then
	\begin{align*}
	c_2 \theta(X)\ad \tensor{I}{U_2^*U_1}\theta(X)c_1 = 
	 \left(\tensor{I}{U_2\ad}\ph(Y)c_2\right)\ad\left(\tensor{I}{U_1}\ph(X)c_1\right).
	\end{align*}
	%Viewed as a relation between maps varying over the pairs $(U, c)$, the above relation is a Grammian, and so we can the standard lurking isometry argument.
	%Because we restrict out attention to the vector spaces $\vecspace^\ph_X$ and $\vecspace^\theta_X$, the isometry is completely defined, and this implies that there exists a unique unitary operator $\tilde{U}_X$, depending on $X$, such that
	Since these inner products agree, there is a uniquely determined partial isometry
	$\tilde{U}_X: \tensor{\vecspace^\theta_X}{\mathbb{C}^n} \to \tensor{\vecspace^\ph_X}{\mathbb{C}^n}$ so that
	\[
	\tilde{U}_X \tensor{I}{U} \theta(X) c = \tensor{I}{U} \ph(X) c
	\]
	for all $U$ and $c$. 
	Now, note that for $U_1, U_2 \in \mathcal U_n$,
		\[
		\tilde{U}_X \tensor{I}{U_1} \tensor{I}{U_2} \theta(X) c = \tensor{I}{U_1} \tensor{I}{U_2} \ph(X) c.
		\]
	Rearranging this equation gives
		\[
		\tensor{I}{U_1^*}\tilde{U}_X \tensor{I}{U_1} \tensor{I}{U_2} \theta(X) c =  \tensor{I}{U_2} \ph(X) c.
		\]
	Note that the uniqueness implies
		$$\tensor{I}{U^*} \tilde{U}_X \tensor{I}{U} = \tilde{U}_X,$$
	and thus $\tilde{U}_X = \tensor{U_X}{I}.$
	\end{proof}

	 The uniqueness of $U_X$ gives the following.
	\begin{proposition}\label{patchAgree}
	Suppose that for free functions $\theta$ and $\ph$ taking a set $D \subset \M$ into $\tensor{\hilbert}{\M}$, we have the relation
	\[
	\theta(X)\ad\tensor{I}{H}\theta(X) = \ph(X)\ad \tensor{I}{H}\ph(X)
	\]
	for all $(X, H) \in T(D)$.
	Let $X, Y \in D$ such that $Y$ dominates $X$. For $U_X$ as defined in Proposition \ref{localunitary},
	\[
	P^\ph_X  U_{Y} P^{\theta *}_X = U_X.
	\]
	\end{proposition}
	\begin{proof}
		Note $P^\ph_X U_{Y} P^{\theta *}_X $ is unitary and thus by uniqueness, $P^\ph_X U_{ Y} P^{\theta *}_X = U_X.$
%		\blue as written these maps don't explicitly make sense. \red They do \blue you've got $P^\theta_X: %H^2_\theta \to \vecspace^\theta_X, U_{X \oplus Y}: \vecspace^\theta_{X \oplus Y} \to \vecspace^\ph_{X \oplus Y},$ and $P_X^\ph: H^2_\ph \to %\vecspace_X^\ph$. should you add that $\vecspace_X^\theta \subset \vecspace_{X\oplus Y}^\theta$ somewhere (which is just by virtue of %choosing the obvious $U$ and $c$) and include a statement that $U_{X\oplus Y}$ acts in the obvious way? Or is it %better to state that we're sort of compressing the action of that function to the ``first slot''? \red there are no %slots, only algebra "domination phenomenon" - This is a lattice morphism from algebras generated by X to %projections on a hilbert space in some cannonical way \black
	\end{proof}

	\begin{theorem}[Lurking isometry argument for linear forms]\label{lurk}
	Let $\theta$ and $\ph$ be free functions on a set $D \subset \M^d$ taking values in $\tensor{\hilbert}{\M}$.   If for all $(X,H) \in T(D)$,
	\[
	\theta(X)\ad\tensor{I}{H}\theta(X) = \ph(X)\ad \tensor{I}{H}\ph(X)
,	\]
	then there exists
	an isometry $U: H^2_\theta \to H^2_\ph$ such that for all $X \in D$,
	\[
	\tensor{U}{I} \theta(X) = \ph(X).
	\]
	\end{theorem}
	\begin{proof}
		The direct limit of the $U_X$ with respect to inclusion of domains
			$$\hat{U}: \bigcup_{X\in D} \vecspace^{\theta}_X \rightarrow \bigcup_{X\in D} \vecspace^{\ph}_X$$
		is well defined by \ref{patchAgree}.
		Furthermore, $\hat{U}$ extends as an isometry $U: H^2_\theta \to H^2_\ph$ since the domain
		of $\hat{U}$ is dense in $H^2_\theta$ by Proposition \ref{denseSpanGen}. Since $P^{\ph}_X U P^{\theta *}_X = U_X,$ $U$ satisfies the required properties.
%		\blue I think this is pretty vague as stands. you don't want to use any of the language of direct limits, %even loosely? \red Maybe, but theres a pretty standard ordering on projections / subspaces \blue If someone looks %at a pile of operators, why are they going to believe that they are looking at a direct system that can be limited %to produce an object? it's also not immediately clear to the reader why the $U$ produced between the Hardy spaces %bears on the mapping between the functions, which are on the Hilbert space. \black
	\end{proof}
\begin{comment}
	 We will also use the following reduction of
	the lurking isometry argument.
	\begin{corollary}[Lurking isometry argument for linear forms]
		Let $\theta$ and $\ph$ be free functions on a set $D \subset \M^d$ taking values in $\tensor{\hilbert}{\M}$.   If for all $X \in D$
	\[
	\theta(X)\ad\tensor{I}{H}\theta(X) = \ph(X)\ad \tensor{I}{H}\ph(X),
	\]
	then there exists
	an isometry $U: H^2_\theta \to H^2_\ph$ such that for all $X \in D$,
	\[
	\tensor{U}{I} \theta(X) = \ph(X).
	\]
	\end{corollary}
	\begin{proof}
		Let $X, Y \in D_n.$
		Let $$W =
		\bpm		
			X & 0 \\
			0 & Y		
		\epm.$$
		Let $$H =
		\bpm
			0 & I_n \\
			I_n & 0		
		\epm.$$
		Note that the relation
		$$H\theta(W)\ad\tensor{I}{H}\theta(W) = H\ph(W)\ad \tensor{I}{H}\ph(W)$$
		implies that 
			$\theta(Y)\ad\theta(X) = \ph(Y)\ad\ph(X)$
		and thus we are done.
	\end{proof}

	The existence of the lurking isometry in the free context will allow us to approach the manipulation and construction of models in a similar fashion to the usual commuting case.
\end{comment}

\section{L\"owner's theorem} \label{LOWNER}
	Let $f: (a,b) \rightarrow \mathbb{R}.$
	We say $f$ is \emph{matrix monotone} if
		$$A \leq B \Rightarrow f(A) \leq f(B).$$
	In 1934, L\"owner \cite{lo34} showed that if $f$ is matrix monotone, then $f$ analytically continues to
	$\Pi_1 \cup (a,b)$ so that
	$f:\Pi_1 \cup (a,b) \rightarrow \overline{\Pi_1}.$

	In general, we define \emph{locally monotone} functions as follows. This definition agrees with classical monotonicity on convex sets since
		$$f(X) - f(Y) = \int_0^1 Df(X + t(X - Y))[X-Y] \, \dd t$$
	by the fundamental theorem of calculus.
	\begin{definition}
		A real analytic free function $f: D \rightarrow \R$  is \emph{locally monotone} if
		$$H\geq 0 \Rightarrow Df(X)[H] \geq 0$$
		for all $(X,H) \in T(D)$.
	\end{definition}
	%The assumption of analyticity can be dropped to differentiability can be dropped to
	%one continuous derivative since monotone functions are analytic as a consequence of
	%the theory of commutative matrix monotone functions \cite{amyloew}. We eschew these issues
	%in favor of brevity.

	The following definition codifies the extension property in the free setting.
	\begin{definition}
		A real analytic free function $f: D \rightarrow \R$  has a \emph{L\"owner extension} if there is a
		continuous free function $F: \Pi^n \cup D \rightarrow \overline{\Pi}$ such that $F|_D= f.$
	\end{definition}
	
	Our goal is to give a version of L\"owner's theorem for the noncommutative functional calculus.
	\begin{theorem}\label{lowTheorem}
		A real analytic free function $f: D \rightarrow \R$ is locally monotone if and only if $f$ has a L\"owner extension. 
	\end{theorem}

\subsection{The Hamburger model}	
	Our first kind of model, the \emph{Hamburger model} gives information about the derivative.
	\begin{definition}
		Let $f: D \rightarrow \RR.$ A \emph{Hamburger model} for $f$ is a list of $d$ real free functions
		$u_i:D \rightarrow \tensor{\mathcal{H}}{\RR}$ such that for all $(X,H) \in T(D),$
			$$Df(X)[H] =\sum_i u_i(X)^*\tensor{I}{H_i}u_i(X).$$
	\end{definition}	
	
	Our second kind of model gives nonlocal data. This has been analyzed in the commutative case on polydisks by Ball and Bolotnikov \cite{babo02}.
	\begin{definition} \label{boundarynev}
		Let $f: D \rightarrow \RR.$ A
		\emph{boundary Nevanlinna model} for $f$ is a list of $d$ real free functions
		$u_i:D \rightarrow \tensor{\mathcal{H}}{\RR}$ such that for all $X, Y$ of the same dimension,
			$$f(X)- f(Y)^* =\sum_i u_i(Y)^*\tensor{I}{X_i-Y_i^*}u_i(X)$$
		and for all $(X,H) \in T(D)$,
			$$Df(X)[H] =\sum_i u_i(X)^*\tensor{I}{H_i}u_i(X).$$
	\end{definition}
	The Hamburger model and the boundary Nevanlinna model are related in these sense that if we have one, we can obtain the other via the relations of free analysis. These can be explicitly computed and are essentially equivalent up to isometry by the lurking isometry for linear forms. We discuss these
	computations in Section \ref{hankel}.
	
	The following expands Theorem \ref{lowTheorem} to give the actual strategy for proof.
	\begin{theorem}\label{lownerfull}
		Let $f: D \rightarrow \R$ be a real analytic free function.
		The following are equivalent:
		\begin{enumerate}
			\item $f$ is locally monotone,
			\item $f$ has a Hamburger model,
			\item $f$ has a boundary Nevanlinna model,
			\item $f$ has a L\"owner extension.
		\end{enumerate}
	\end{theorem}
	
	We regard the implication $(1 \Rightarrow 2)$ to be the novel part of the proof. The implication
	$(2 \Leftarrow 1)$ holds \emph{a fortiori} because of the form of the Hamburger model. We devote the rest
	of this section to proving the simpler parts, and will prove $(1 \Rightarrow 2)$ as Section \ref{construction}.
	The implication $(1\Rightarrow 4)$ in Theorem \ref{lownerfull} is proven as Theorem \ref{nevrawr}.
	
	The following lemma proves $(2\Leftrightarrow 3)$ in Theorem \ref{lownerfull}.
	\begin{lemma} \label{nevham}
		A real free analytic function $f: D \rightarrow \R$ has a Hamburger model if and only if $f$ has a boundary Nevanlinna model.
	\end{lemma}
	\begin{proof}
		The reverse implication holds by definition.
	
		Suppose $f$ has a Hamburger model $u.$
		That is, $$Df(X)[H] = \sum_i u_i(X)^*\tensor{I}{H}u_i(X).$$
		So,
		$$Df\bpm
			X & \\
			 & Y
		\epm
		\bbm
			 & X-Y \\
			X-Y  & 
		\ebm = \sum_i u_i\bpm
			X & \\
			 & Y
		\epm^*\tensor{I}{
		\bbm
			 & X_i-Y_i \\
			X_i-Y_i  & 
		\ebm}
		u_i\bpm
			X & \\
			 & Y
		\epm.$$
		
		Via the formula in Proposition \ref{derUnit},
		$$\bbm
			 & f(X)-f(Y) \\
			f(X)-f(Y)  & 
		\ebm =
		\sum_i 
		\bbm
			u_i(X) & \\
			 & u_i(Y)
		\ebm^*
		\tensor{I}{
		\bbm
			 & X_i-Y_i \\
			X_i-Y_i  & 
		\ebm}
		\bbm
			u_i(X) & \\
			 & u_i(Y)
		\ebm.$$
		
		Multiplying on the second slot,
		$$\bbm
			 & f(X)-f(Y) \\
			f(X)-f(Y)  & 
		\ebm = \sum_i\bbm
			 & u_i(X)^*\tensor{I}{X_i-Y_i}u_i(Y) \\
			u_i(Y)^*\tensor{I}{X_i-Y_i}u_i(X) & 
		\ebm.$$
		
		This implies
		$$f(X)-f(Y) = \sum_i u_i(Y)^*\tensor{I}{X_i-Y_i}u_i(Y).$$
		
		Thus,
		$$f(X)-f(Y)^* = \sum_i  u_i(Y)^*\tensor{I}{X_i-Y_i}u_i(Y).$$
	\end{proof}

	The following lemma proves $(4\Rightarrow 1)$ in Theorem \ref{lownerfull}.
	\begin{lemma}
		If a real free analytic function  $f: D \rightarrow \R$ has a L\"owner extension, then $f$ is locally monotone.  
	\end{lemma}
	\begin{proof}
		Suppose $f$ has a L\"owner extension and $f$ is not monotone.
		Then, there is a point $X$ and a positive semidefinite $H$ such that $D(X)[H]$ is not positive semidefinite.
		Since $$f(X+itH) = f(X) +  itD(X)[H] + O(t^2),$$
		$$\text{Im } f(X+itH) = tD(X)[H] + O(t^2)$$
		which for small $t\geq 0$ is not positive semidefinite. This is a contradiction.
	\end{proof}

\subsection{The Hamburger model construction} \label{construction}
	We now begin the construction of a Hamburger model.
	\begin{comment}
	\begin{lemma} \label{monocompos}
	Suppose	$f: D \rightarrow \RR$ is locally monotone on $D.$ Then $Df(X)[H]$ is a \emph{completely positive map} in each $H_i$, that is .
	\end{lemma}
	
	\end{comment}
	The following is a reduction of the Choi-Kraus theorem which gives the
	raw data used in the construction locally.
	\begin{lemma} \label{finitehamburger}
		Suppose	$f: D \rightarrow \RR$ is locally monotone on $D.$
		For any point $X\in D_n$ there are real free functions $u_{ij}$ such that for all $H \in \M_n^d$,
			$$Df(X)[H] = \sum_{i} \sum_{j} u_{ij}(X)^*H_iu_{ij}(X).$$ 
	\end{lemma}
	\begin{proof}
		Suppose $f$ is locally monotone. Let $X \in D\cap \M^d_n.$ Without loss of generality, we let $X = \bigoplus X_l$ where $X$ generates the algebra $\bigoplus M_{n_l}(\mathbb{C})$ where $\sum n_l = n.$
		(This follows from Lemmas \ref{dominationDer} and \ref{domReduction}.)
		%By Lemma \ref{monocompos},
		Note $Df(X): \M_n^d \rightarrow \M_n$ is completely positive in each coordinate
		since the extension of
		$Df(X)$ to $\M_{n}^d \otimes \M_k$ via Formula \eqref{composten}
		is given by $Df(X\otimes I_n)$ by Proposition \ref{tenDer} which is positive by the assumption of local monotonicity.
		By the Choi-Kraus theorem\cite{bha07}, $$Df(X)[H] = \sum_{i} \sum_{j} V_{ij}^*H_iV_{ij}.$$
		
		We now show that the $V_{ij}$ are in the algebra generated by $X.$ That is, they are free polynomial functions of $X.$ 
		Let $P^l$ be the projection onto the $l$-th component of $X.$
		Let $P^l_i$ be a tuple that equals $P^l$ on the $i$-th coordinate and $0$ elsewhere.
		Consider $Df(X)[P^l_i].$
			$$Df(X)[P^l_i] = \sum_{j} V_{ij}^*P^lV_{ij}.$$
		Block decompose $$V_{ij} = \sum_{l,m} P^lV_{ij}P^m.$$
		So,
			$$Df(X)[P^l_i] = \sum_{j} \sum_m \sum_n P^mV_{ij}^*P^lV_{ij}P^n.$$
		So, in the block decomposition of $Df(X)[P^l_i]$ the $(m,m)$ entry is $$ \sum_{j} P^mV_{ij}^*P^lV_{ij}P^m.$$
		However, since $Df(X)[H]$ is a free function,
			$$Df(X)[P^l_i] = \sum_{j} P^lV_{ij}^*P^lV_{ij}P^l$$
		So, if $l \neq m$ the $(m,m)$ entry is $0.$ That is, $P^mV_{ij}^*P^lV_{ij}P^m = 0.$ Thus, if $l \neq m$
		$P^lV_{ij}P^m = 0.$ This implies that 
			$$V_{ij} = \sum_l P^lV_{ij}P^l.$$
		
		Since $X$ generates $\bigoplus M_{n_l}(\mathbb{C}),$ $V_{ij}$ is in the algebra generated by $X.$
		Thus, each $V_{ij}$ is in the algebra generated by $X$ so there are free polynomials $u_{ij}$ such that $u_{ij}(X) = V_{ij}.$
	\end{proof}
	
	\begin{definition}
		A \emph{global monomial basis vector} for $D$ is a free function $m$ on $D$ given by the formula $m_X = (c_IX^I)_I$ such that $\|m_X\|$
		is locally bounded on $D$ and each $c_I > 0.$
	\end{definition}
	
	We now use free coefficient Hardy space methods to establish local uniqueness of the Hamburger model. We refine the
	the raw data obtained in Lemma \ref{finitehamburger} into a canonical object.
	\begin{lemma} \label{unique}
		Let $m$ be a global monomial basis vector.
		If $f: D \rightarrow \RR$ is locally monotone on $D,$
		for any matrix tuple $X \in D_n$ there are unique finite rank operators $A^i_{X} \in \mathcal B(H_m^2) \geq 0$ such that if at a tuple of operators $B_i \in \B(H_m^2)$
			$$Df(X)[H] = \sum_{i} m_X^* \tensor{B_i }{ H_i}m_X$$ 
		for all $H \in \M_n^d$, then $P^m_XB_iP^m_X = A^i_X.$ % P(X) is the projection onto some special space
	\end{lemma}
	\begin{proof}
		Note by Lemma \ref{finitehamburger}, there are polynomials $u_{ij}$ such that
			$$Df(X)[H]= \sum_{i} \sum_{j} u_{ij}(X)^*H_iu_{ij}(X).$$
		Define $u_i$ to be the function given by $(u_{ij}(X))_{j}$ as a column vector.
		So,
			$$Df(X)[H]= \sum_{i}  u_{i}(X)^* \tensor{I}{H_i}u_{i}(X).$$
		Note, since the entries of $u_i$ are polynomials, there are bounded finite rank operators $K_i$ so that
		$u_i(X) = \tensor{K_i}{I}m_X$.
		So
			$$Df(X)[H]= \sum_{i}  m_X^* \tensor{K_i^*K_i}{H_i}m_X.$$
		Define $A^i_X = P^m_XK_i^*K_iP^m_X.$
			$$Df(X)[H] = \sum_{i} m_X^* \tensor{A^i_X}{H_i}m_X.$$
		
		Suppose, $B_i$ satisfies $$Df(X)[H] = \sum_{i} m_X^* \tensor{B_i}{H_i} m_X$$
		and $P^m_XB_iP^m_X=B_i.$
		So if $U_1, U_2 \in \mathcal U_n$
			$$0 = m_X^* \tensor{A^i_X - B_i}{U_1^*U_2}m_X = m_X^* \tensor{I}{U_1^*}\tensor{A^i_X - B_i}{I}\tensor{I}{U_2}m_X$$
		So, $\ip{(A^i_X-B_i)u}{v}=0$ for any $u, v \in \vecspace^m_X $ and is thus $0.$
	\end{proof}
	
	By patching these together, we obtain the Hamburger model.
	\begin{lemma} %\label{uniqueham}
		If $f: D \rightarrow \RR$ is locally monotone on $D,$
		then $f$ has a Hamburger model.
		%For any point $X$ there is are functions $u_{i}$ and operators $A_i \geq 0$ such that
		%	$$Df(X)[H] = \sum_{i} u_{i}(x)^* A_i \oplus H_iu_{i}(x).$$ 
	\end{lemma}
	\begin{proof}
		By Lemma \ref{unique}, for each $X$ there is a unique $A^i_X$ such that
		$Df(X)[H] = \sum_i m_X^*\tensor{A^i_X}{H_i} m_X$ and $P^m_XA^i_XP^m_X = A^i_X.$
		
		Let $q^i(u,v)$ be the direct limit of the semidefinite sesquilinear forms
		$q_X^i(u,v) = \ip{A^i_X u}{v}$ defined on $\vecspace = \bigcup_{X\in D} \vecspace^m_X.$
		Let $K_i$ be the Hilbert space formed by completing the quotient $\vecspace / \text{ker } q^i(\cdot,\cdot).$
		Let $T_i: \vecspace \rightarrow K_i$ be the inclusion map.
		Define $u_i = \tensor{T_i}{I} m_X.$
		Now $u_i$ form a Hamburger model for $f.$
	\end{proof}

\subsection{The localizing matrix construction of the Hamburger model} \label{hankel}
	Let $f(X) =\sum_I c_IX^I$ be a power series.
	\emph{The $x_k$-localizing matrix of coefficients}
	is the infinite matrix with rows and columns indexed by monomials $(c_{I^*x_kJ})_{I,J}.$
	In the section, we will show that if $f$ is monotone, then the $x_k$-localizing matrix of coefficients must
	be positive semidefinite.
	This application mirrors the use of classical Hankel matrices
	in the study of the Hamburger moment problem \cite{nev22, pel02}. Localizing matrices have been
	used to study multivariate moment problems\cite{curto}, and more recently to study noncommutative convex hulls\cite{freelasserre}.
	
	The following gives a condition for a free power series to be uniformly and absolutely convergent.
	\begin{lemma}\label{powerSeriesConverge}
		Let $\epsilon >0.$
		Suppose a series in $d$ noncommuting variables $\sum_I c_IX^I$ is convergent for all $\|X\|<d+\epsilon.$
		Then $\sum_I c_IX^I$ is absolutely and uniformly convergent for all $\|X\|<1.$
		Furthermore, there is an $N$ such that if $|I| \geq N,$
			$$|c_I| \leq \left(d + \frac{\epsilon}{2}\right)^{-|I|}.$$
	\end{lemma}
	\begin{proof}
		Note that
			$$\lim_{n\rightarrow \infty} \max_{|I| = n} \|c_IX^I\| = 0$$
		for all $\|X\|<d+\epsilon.$
		Substituting in the tuple $(d+\epsilon/2, \ldots, d+\epsilon/2)$ gives that
			$$\lim_{n\rightarrow \infty} \max_{|I| = n}  |c_I| (d+\epsilon/2)^{|I|} = 0$$
		which implies that for large $n,$ $|c_I|(d+\epsilon/2)^{|I|} \leq 1$ which implies the claim.
	\end{proof}
	
	We will now establish that for power series that converge on large enough sets, the $x_k$-localizing matrices are compact,
	which will be useful in establishing formulas for the derivative of a real free power series.
	\begin{lemma}\label{heltonHankelXCompact}
		If $f(X) =\sum_I c_IX^I$ is a real analytic locally monotone free function for $\|X\|<d+\epsilon$,
		then for each $k,$ the $x_k$-localizing matrix of coefficients $(c_{I^*x_kJ})_{I,J}$
		is compact.
	\end{lemma}
	\begin{proof}
		The compactness of each $(c_{I^*x_kJ})_{I,J}$ follows from the decay of the entries given in Lemma \ref{powerSeriesConverge}.
		That is, if $E_{I,J}$ is the infinite matrix with entry $1$ in the $(I, J)$-th slot and zero elsewhere, then
			$$(c_{I^*x_kJ})_{I,J} = \sum^{\infty}_{n=1}\sum_{|Ix_kJ|=n} c_{Ix^kJ}E_{I,J}$$
		is a convergent formula in the norm topology since it is Cauchy via the estimate
		(relying on the combinatorial observation that the total number of words of
		length $n$ in $d$ letters is $d^n$ and the estimate in Lemma \ref{powerSeriesConverge}.)
			\begin{align*}
			\|\sum^N_{n=M} \sum_{|Ix_kJ|=n}  c_{Ix^kJ}E_{I,J}\|
			&\leq \sum^N_{n=M} \sum_{|Ix_kJ|=n} |c_{Ix^kJ}|\\
			&= \sum^N_{n=M} \sum_{|Ix_kJ|=n} (d+\epsilon/2)^{-|Ix_kJ|}\\
			&\leq \sum^N_{n=M} \sum_{|Ix_kJ|=n} (d+\epsilon/2)^{-n}\\
			&\leq \sum^N_{n=M} d^{n} (d+\epsilon/2)^{-n}\\
			&= \sum^N_{n=M}  (1+\frac{\epsilon}{2d})^{-n}\\
			&\leq  \sum^\infty_{n=M}  (1+\frac{\epsilon}{2d})^{-n} \\
			&=  \frac{(1+\frac{\epsilon}{2d})^{-M}}{\frac{\epsilon}{2d}} \rightarrow 0.
			\end{align*}
		Thus,
		since $(c_{I^*x_kJ})_{I,J}$ is well-approximated by finite rank operators,
		$(c_{I^*x_kJ})_{I,J}$ is compact \cite[Theorem 4.4]{con97}.
	\end{proof}
	
	We will need the following elementary fact about the derivative.
	\begin{proposition}\label{derivFormula}
		Let $\epsilon >0.$
		Suppose a series in $d$ noncommuting variables $\sum_I c_IX^I$ is convergent for all $\|X\|<d+\epsilon.$
		The derivative of $f$ at $X \in D_n$ in the direction $H \in \M_n^d$ is, for $\|X\|\leq \frac{1}{d}$, given by the formula
			$$Df(X)[H] = \sum_k m_X^* \tensor{(c_{I^*x_kJ})_{I,J}}{H_k} m_X.$$
	\end{proposition}
	\begin{proof}
		By Lemma \ref{heltonHankelXCompact} $(c_{I^*x_kJ})_{I,J}$ is compact and thus it will be sufficient to show
		that, for the functions $g_K(X) = X^K,$ $$Dg_K(X) = \sum_k m_X^* \tensor{(\chi(I^*x_kJ=K))_{I,J}}{H_k} m_X.$$
		where $\chi$ is the indicator function. Since
		$$\sum_k m_X^* \tensor{(\chi(I^*x_kJ=K))_{I,J}}{H_k} m_X = \sum_k X^J\chi(I^*x_kJ=K)H_kX^I$$
		and the right hand side of the preceding equation is the derivative by the product rule, we are done.
	\end{proof}

	The following gives a characterization of monotone functions in terms of their power series, similarly to
	Nevanlinna's solution to the Hamburger moment problem\cite{nev22}.
	\begin{theorem}\label{heltonHankelX}
		If $f(X) =\sum_I c_IX^I$ is a real analytic locally monotone free function for $\|X\|<d+\epsilon$,
		then for each $k,$ the $x_k$-localizing matrix of coefficients $(c_{I^*x_kJ})_{I,J}$
		is positive semidefinite and compact.
	\end{theorem}
	\begin{proof}
		Let $m_X = (X^I)_I.$
		Note, for $\|X\| < \frac{1}{d},$ $m_X$ is bounded and
			\beq \label{heltonHankel}
				Df(X)[H] = \sum_k m_X^* \tensor{(c_{I^*x_kJ})_{I,J}}{H_k} m_X
			\eeq
		via Lemma \ref{derivFormula}.
		
		By Lemma \ref{unique}
			$$P_X (c_{I^*x_kJ})_{I,J} P_X$$
		is positive semidefinite, and thus since $\vecspace = \cup_X \vecspace_X$ is dense in
		$H_d^2$ by Proposition \ref{denseSpan}, $(c_{I^*x_kJ})_{I,J}$ is positive semidefinite.
	\end{proof}
	We remark that expressions of the form \eqref{heltonHankel} have been used to analyze the
	noncommutative Hessian to obtain results on convex functions \cite{helmconvex04}. 
	
	Reinterpreting Theorem \ref{heltonHankelX}, we immediately obtain a model for a locally monotone function.
	\begin{corollary}
		Let $f(X) =\sum_I c_IX^I$ be a  real analytic locally monotone free function for $\|X\|<d+\epsilon$.
			The Hamburger model for $f$ on $\|X\|<\frac{1}{d}$ is given by the formula	
			$$Df(X)[H] = \sum_k m_X^* \tensor{(c_{I^*x_kJ})_{I,J}^{1/2}}{I} \tensor{I}{H_k} \tensor{(c_{I^*x_kJ})^{1/2}_{I,J}}{I} m_X.$$
	\end{corollary}
	
	\begin{lemma} \label{analCont}
		Let $f: D \rightarrow \RR$ be a  real analytic locally monotone free function.
		For any $X \in D$, the Hamburger model for $f$ analytically continues to a free domain containing
		$X.$
	\end{lemma}
	\begin{proof}
		Let $f$ have the Hamburger model
			$$Df(X)[H] = \sum_i u_i(X)^* \tensor{I}{H_i} u_i(X).$$
			
		Let $X_0 \in D\cap \RR^d_n.$
		Let $F_1, \ldots F_{dn^2}$ be tuples of positive semidefinite matrices which span $\RR^d_n$
		large enough so that
		the function of $dn^2$ variables
		 $$h_{i}(X) = (\flattensor{e_i}{I})^*f(\flattensor{X_0}{I} + X_F)(\flattensor{e_i}{I}) $$
		is analytic for all $\|X\| \leq dn^2 + \epsilon.$
		Note that each $h_i$ is a monotone function. 
		Thus for each $i$ each
			$$Dh_i(X)[H] = \sum_k m_X^* \tensor{(c^i_{I^*x_kJ})_{I,J}^{1/2}}{I} \tensor{I}{H_k} 			
			\tensor{(c^i_{I^*x_kJ})^{1/2}_{I,J}}{I} m_X.$$
		Note that
			$$(\flattensor{e_j}{I})^*Df(\flattensor{X_0}{I} + X_F)[H_F](\flattensor{e_j}{I}) =  Dh_j(X)[H].$$
		So,
			\begin{align*}
			&\sum_j \sum_i (\flattensor{e_j}{I})^*u_i(\flattensor{X_0}{I} + X_F)^*
			\tensor{I}{(H_F)_i} u_i(\flattensor{X_0}{I} + X_F)(\flattensor{e_j}{I}) \\
			&=
			\sum_j \sum_k m_X^* \tensor{(c^j_{I^*x_kJ})_{I,J}^{1/2}}{I}	\tensor{I}{H_k} 		
			\tensor{(c^j_{I^*x_kJ})^{1/2}_{I,J}}{I} m_X.
			\end{align*}
		Expanding out the frame gives
			\begin{align*}
			&\sum_k \sum_j \sum_i (\flattensor{e_j}{I})^*u_i(\flattensor{X_0}{I} + X_F)^*
			\tensor{I}{(\flattensor{F_k}{H})_i} u_i(\flattensor{X_0}{I} + X_F)(\flattensor{e_j}{I}) \\ 
			&= \sum_j \sum_k m_X^* \tensor{(c^j_{I^*x_kJ})_{I,J}^{1/2}}{I}	\tensor{I}{H_k} 		
			\tensor{(c^j_{I^*x_kJ})^{1/2}_{I,J}}{I} m_X.
			\end{align*}
		Thus, by the lurking isometry argument for linear forms, there is a unitary $U$ such that
		$\tensor{U}{\flattensor{I}{I}}$
	    takes
	    $$\theta(X) = \bigoplus_{i,j,k} \tensor{I}{(\flattensor{F_k}{I})_i^{1/2}} u_i(\flattensor{X_0}{I} + X_F)(\flattensor{e_j}{I})$$
	    to $$\ph(X) = \bigoplus_{j,k} \tensor{(c^j_{I^*x_kJ})^{1/2}_{I,J}}{I} m_X.$$ Since
	    $\ph(X)$  analytically continues to a neighborhood of $0$ via its formula, so does 
	    $\theta(X).$
	    
	    The above implies each $u_i$ itself must analytically continue, since the value
	    of $\theta(X)$ determines the values of each $u_i(X).$
	    The continuation of $u_i$ a free function by Corollary \ref{commieCor}.
	    Thus, the analytic continuation of the Hamburger model is given by the formula
	    		$$Df(Z)[H] = \sum_i u_i(Z^*)^* \tensor{I}{H_i} u_i(Z)$$
	    	since this agrees with the Hamburger model on a neighborhood of $X_0.$
	\end{proof}
	
	Thus, we obtain the desired L\"owner extension.	
	Thus, we obtain L\"owner's theorem by a rearrangement argument.
	\begin{theorem}\label{nevrawr}
		Let $D$ be a free domain.
		If $f:D \rightarrow \RR$ is a  real analytic locally monotone free function, then $f$ has a L\"owner extension.
	\end{theorem}
	\begin{proof}
		Let $X_0 \in D.$	
	
		Note that formula
		$$Df(Z)[H] = \sum_i u_i(Z^*)^* \tensor{I}{H_i} u_i(Z)$$
		can be used to derive a Nevanlinna model
		$$f(X) - f(Y)^* = \sum_i u_i(Y)^* \tensor{I}{X_i-Y_i^*} u_i(X).$$
		by Proposition \ref{derUnit}.
		Thus, there is some $D^X$ containing $X_0$ such that $f|_{D^X}$ has an analytic continuation to $\Pi^d$ by Lemma \ref{analCont} via Theorem \ref{ncschur2}.
		
		Note that for $X_0, Y_0 \in D$ there is a L\"owner extension at $X_0 \oplus Y_0$ which induces the
		L\"owner extension at $X_0$ and $Y_0$ and thus must be the same.
	\end{proof}

\section{The Nevanlinna representations} \label{NevSect}
Rolf Nevanlinna characterized the class of Pick functions in terms of three parameters: a real number, a nonnegative real number and a finite positive Borel measure on the real line.
	\begin{theorem}[R. Nevanlinna \cite{nev22}] \label{NevanlinnaRepComp}
	Let $h: \mathbb{H} \to \C$. The function $h$ is a Pick function if and only if there exist $a \in \R, b \geq 0$, and a finite positive Borel measure $\mu$ on $\R$ such that 
	\[
	h(z) = a + bz + \int\!\frac{1 + tz}{t - z}\, \dd\mu(t)
	\]
	for all $z \in \mathbb{H}$. Moreover, for any Pick function $h$, the numbers $a \in \R, b \geq 0$ and the measure $\mu \geq 0$ are uniquely determined.
	\end{theorem}

This representation parametrizes all Pick function, generalizing Theorem \ref{NevanlinnaRep}, which represents Pick functions satisfying the growth condition \eqref{NevanlinnaRepFormulaAsy}.

The two representations given in Theorem \ref{NevanlinnaRepComp} and Theorem \ref{NevanlinnaRep}, which we refer to as Nevanlinna representations, were recently generalized to several variables in the work of Agler, {\mc}Carthy, and Young \cite{amyloew}, and later in Agler, Tully-Doyle, and Young for multivariable Pick functions that are also in the so-called L\"owner class \cite{aty12,aty13}. 

The $d$-variable L\"owner class $\mathcal L_d$ is the set of analytic functions that lift via the functional calculus to act on $d$-tuples of commuting operators with the property that they possess analytic extensions that take $\Pi^d \to \overline{\Pi}$.  In the case of one-variable Pick functions, the L\"owner class is the the entire Pick class by von Neumann's inequality \cite{vonN51, szn-foi}. For two variables, the classes coincide by And\^o's Theorem \cite{and63}. In three or more variables, the L\"owner class is a proper subset of the Pick class \cite{par70} \cite{var74}. The L\"owner class is conformally equivalent to the well-studied Schur-Agler class via a M\"obius transform taking the upper half plane $\mathbb{H}$ to the disk to obtain a map $\mathbb{D}^d \to \mathbb{D}.$

The multivariable Nevanlinna representations discussed in \cite{aty13} partitions the L\"owner class $\mathcal L^d$ into four types depending on asymptotic behavior at infinity. Note that these not only give representations of Pick functions, but also give formulas for constructing new Pick functions.
\begin{definition}
A \emph{positive decomposition} is a collection of positive operators $Y_1, \ldots, Y_d$ on $\hilbert$ so that $\sum_i Y_i = I_\hilbert$, and the operator $z_Y: \C^d \to \hilbert$ is given by 
\[
z_Y = \sum_{i=1}^d z_i Y_i.
\] An \emph{orthogonal decomposition} is a positive decomposition where each $Y_i$ is a projection.
\end{definition}

By the following theorem, every $h$ in $\mathcal L_d$ has a representation of at least one type. Recall that $\mathbb H^d = \Pi_1^d$.
\begin{theorem}[Agler, Tully-Doyle, Young \cite{aty13}]\label{oldnevreps}
Let $h$ be a function defined on $\Pi_1^d$. 
\begin{description}
\item[Type 4]  The function $h \in \mathcal L_d$ if and only if there exist an orthogonally decomposed
Hilbert space $\hilbert$ and a vector $v \in \hilbert$ such that 
\[
h(z) = \ip{M(z)v}{v},
\]
where $M(z)$ is the matricial resolvent as defined in Proposition 3.1 in \cite{aty13}.
\item[Type 3] There exist $a \in \R$, a Hilbert space $\hilbert$, a self-adjoint operator $A$ on $\hilbert$, a vector $v \in \hilbert$, and a positive decomposition
 $Y$ of $\hilbert$ so that
\[
h(z) = a + \ip{(1-iA)(A- z_Y)\inv(1 + z_Y A)(1 - iA)\inv v}{v}
\]
if and only if $h \in \mathcal L_d$ and
\[
\liminf_{s\to\infty} \frac{1}{s} \IM h(is, \ldots, is) = 0.
\]
\item[Type 2]  There exist $a \in \R$, a Hilbert space $\hilbert$, a self-adjoint operator $A$ on $\hilbert$, a vector $v \in \hilbert$, and a positive decomposition $Y$ of $\hilbert$ so that
\[
h(z) = a + \ip{(A - z_Y)\inv v}{v}
\]
if and only if $h \in \mathcal L_d$ and
\[
\liminf_{s \to \infty} s \IM h(is, \ldots, is) < \infty.
\]
\item[Type 1]  There exist a Hilbert space $\hilbert$, a self-adjoint operator $A$ on $\hilbert$, a vector $v \in \hilbert$, and a positive decomposition $Y$ of $\hilbert$ so that 
\[
h(z) = \ip{(A - z_Y)\inv v}{v}
\]
if and only if $h \in \mathcal L_d$ and
\[
\liminf_{s \to \infty} s\abs{h(is, \ldots, is)} < \infty.
\]
\end{description}
\end{theorem}

In the following sections, we will extend this theorem to characterize free Pick functions. 

\subsection{The free Herglotz representation formula}

A \emph{Herglotz function} is a holomorphic function $h$ on $\D$ with $\RE h \geq 0.$ Herglotz functions were characterized by C. Carath\'eodory in \cite{car07} and given an integral representation by Gustav Herglotz in \cite{herg11}. J. Agler generalized this representation for functions in the $d$-variable strong Herglotz class, a conformal equivalent of the Schur-Agler class, in \cite{ag90}.  %(It is often convenient to normalize Herglotz functions so that $h(0) = 1$).

To define free Herglotz functions, we need an analogue of the right halfplane. Denote by $\Psi \subset \M$ the right matrix polyhalfplane
\[
	\Psi = \{X \in \M | \RE X = \frac{1}{2}(X + X\ad) > 0\}.
\]
Recall that $\B^d$ is the set of $d$-tuples of strict contractions. Say that $h$ is a \emph{free Herglotz function} if $h$ is a free holomorphic function from $\B^d \to \Psi$. 

The representation of a classical Herglotz function $h$ with $h(0) = 1$ is given by a probability measure $\mu$ on the unit circle so that
	$$h(z) = \int_0^{2\pi} \frac{1 + e^{-i\theta}}{1 - e^{-i\theta}} \,\dd\mu(\theta).$$
The following theorem gives an analogous formula in the free case. The noncommutative Herglotz representation was originally proved by
G. Popescu in \cite[Theorem 3.1]{popescumajorant}.
We express the representation in terms of the geometry from the commutative case in Agler, Tully-Doyle and Young's proof of the Nevanlinna representations for L\"owner functions in several variables\cite{aty13}.

\begin{theorem}[Popescu \cite{popescumajorant}]
	Let $h$ be a free holomorphic function with $h(0) = 1$. The function $h$ is a free Herglotz function, that is $h:\B^d \to \Psi$, if and only if there exist a Hilbert space $\hilbert_d = \oplus_{i=1}^d \hilbert$, a unitary operator $U$ on $\hilbert_d$ and an isometry $V: \C \to \hilbert_d$ so that 
		\beq \label{herglotz}
		h(X) = \tensor{V\ad}{I} \left( I + \tensor{U}{I}\delta(X)\right)\left( I - \tensor{U}{I}\delta(X)\right)\inv \tensor{V}{I}, 
		\eeq
	where 
		$$\delta(X): \M^d \to \tensor{\hilbert_d}{\M} = \bigoplus_{i=1}^d \tensor{I}{X^i}.$$
\end{theorem}

\begin{corollary} \label{altherg}
If $h$ is a free Herglotz function on $\B^d$, then there exists a Hilbert space $\hilbert_d = \oplus_i \hilbert_i$, a real constant $a$, a unitary operator $L \in \B(\hilbert_d)$, and a vector $v \in \hilbert_d$ such that for all $X \in \B^d$, 
\beq  \label{altherg1}
h(X) = -i\tensor{a}{I} + \tensor{v\ad}{I}\left(\tensor{L}{I} - \delta(X)\right)\inv\left(\tensor{L}{I} + \delta(X)\right)\tensor{v}{I},
\eeq
where $\delta(X) = \bigoplus_{i=1}^d \tensor{I}{X_i}$.

Conversely, any $h$ defined by an equation of the form \eqref{altherg1} is a free Herglotz function.
\end{corollary}

Connections between Pick functions, Herglotz functions, and Schur functions are given by the \emph{Cayley transform}. The particular Cayley transform given by
\[
	\la = \frac{z - 1}{z + 1}, \,\, z = \frac{1 + \la}{1 - \la}
\]
is a conformal mapping between the right half-plane and the disk. G. Popescu showed that the noncommutative Cayley transform is a well-defined bijection between free Herglotz functions and free Schur functions in \cite[Section 1]{po08}. Via the Cayley transform, we will use representation of free Herglotz functions to derive our Nevanlinna representations.
\subsection{The structured resolvents}

The expression $f(z) = (t-z)\inv$ in Theorem \ref{NevanlinnaRep} suggests the resolvent operator
$R(Z) = (A - z)\inv$. We will present, in the following section, four representations based on the following \emph{structured resolvents} (properties of structured resolvents are discussed at length in \cite{aty13}).

\begin{definition}
Let $\hilbert$ be a Hilbert space. A \emph{positive decomposition} $Y$ of $\hilbert$ is a collection of positive operators $Y_1, \ldots, Y_d$ summing to $I.$ An \emph{orthogonal decomposition} $P$ of $\hilbert$ is a collection of pairwise orthogonal projections $P_1, \ldots, P_d$ on $\hilbert$ summing to $I.$
\end{definition}

\begin{definition}[structured resolvents]\hfill
\begin{description}
\item[Type 2] \label{type2res}
Let $A$ be a closed densely defined self-adjoint operator on a Hilbert space $\hilbert$ and let $Y$ be a positive decomposition of $\hilbert$. The \emph{structured resolvent of type $2$} corresponding to $Y$ is the function from $\Pi^d \to \tensor{\B(\hilbert)}{\M}$ given by
\beq \label{type2formula}
	M_2(Z) =  \left(\tensor{A}{I} - \delta_Y(Z)\right)\inv,
\eeq
where 
	$$\delta_Y(Z) = \sum_{i=1}^d \tensor{Y_i}{Z_i}.$$
\item[Type 3] \label{type3res}
Let $A$ be a closed, densely defined, self-adjoint operator on a Hilbert space $\hilbert$, and let $Y$ be positive decomposition of $\hilbert$. The \emph{structured resolvent of type $3$} corresponding to $Y$ is the function from $\Pi^d \to \tensor{\B(\hilbert)}{\M}$ given by
\beq \label{type3formula}
M_3(Z) = \tensor{(I - iA)}{I}\left(\tensor{A}{I} - \delta_Y(Z)\right)\inv \left(I + \delta_Y(Z)\tensor{A}{I}\right) \tensor{(I - iA)\inv}{I},
\eeq
where 
	$$\delta_Y(Z) = \sum_i \tensor{Y_i}{Z_i}.$$
\item[Type 4] \label{type4res}
Let $\hilbert$ be a Hilbert space decomposed orthogonally as $\hilbert = \N \oplus \K$. Let $A$ be a densely defined, self-adjoint operator on $\K$ with domain $\mathcal D(A)$, and let $P$ be an orthogonal decomposition of $\hilbert$. The structured resolvent of type $4$ resolvent is a function from $\Pi^d \to \tensor{\B(\hilbert)}{\M}$ of the form
\begin{align}
M_4(Z) &= \tensor{\bbm -i & 0 \\ 0 & I - iA \ebm}{I}\left( \tensor{\bbm I & 0 \\ 0 & A \ebm}{I} -\delta_P(Z)\tensor{\bbm 0 & 0 \\ 0 & I \ebm}{I} \right)\inv \times \notag \\ 
&\hspace{1.5in} \left(\delta_P(Z) \tensor{\bbm I & 0 \\ 0 & A \ebm}{I} + \tensor{\bbm 0 & 0 \\ 0& I \ebm}{I}\right)\tensor{\bbm -i & 0 \\ 0 & I-iA \ebm\inv}{I}, \label{type4formula}
\end{align}
where 
	$$\delta_P(Z) = \sum_i \tensor{P_i}{Z_i}.$$
\end{description}

For each $2 \leq i \leq 4$ and each $Z \in \Pi^d$, the expression $M_i(Z)$ is a bounded operator and 
has positive imaginary part, which follows directly from proofs given in \cite{aty13}.
\end{definition}

With the structured resolvents \eqref{type2res}, \eqref{type3res}, and \eqref{type4formula}, we now present representations for free Pick functions that generalize the classical Nevanlinna representations given in Theorems \ref{NevanlinnaRep} and\ref{NevanlinnaRepComp}. 

Recall that the free Pick class $\Pick_d$ consists of analytic functions $h: \Pi^d \to \cc{\Pi^1}$.
\begin{definition}
The following are the representations for functions in $\Pick^d.$
\begin{description}
	\item[Type 4] A Nevanlinna representation of type $4$ of a function $h$ in the free Pick class on $\M^d$ 
	is
\[
h(z) = \tensor{a}{I} + \tensor{v^*}{I} M_4(Z) \tensor{v}{I},
\]
where $M_4(Z)$ is a type $4$ structured resolvent as given in Definition \ref{type4res} with associated Hilbert space $\hilbert$, $v$ is a vector in $\hilbert$,  and $a \in \R$.
	\item[Type 3] A Nevanlinna representation of type $3$ of a function $h$ in the free Pick class on $\M^d$ 
	is
\[
h(z) = \tensor{a}{I} + \tensor{v^*}{I} M_3(Z) \tensor{v}{I},
\]
where $M_3(Z)$ is a type $3$ structured resolvent as given in Definition \ref{type3res} with associated Hilbert space $\hilbert$, $v$ is a vector in $\hilbert$,  and $a \in \R$.
	\item[Type 2]
	A Nevanlinna representation of type $2$ of a function $h$ in the free Pick class on $\M^d$ 
	is
\[
h(Z) = \tensor{a}{I} + \tensor{v^*}{I} M_2(Z) \tensor{v}{I},
\]
where $M_2(Z)$ is a type $2$ structured resolvent as given in Definition \ref{type2res} with associated Hilbert space $\hilbert$, $v$ is a vector in $\hilbert$,   and $a \in \R$.
	\item[Type 1] A Nevanlinna representation of type $1$ of a function $h$ in the free Pick class on $\M^d$ 
	is
\[
h(Z) = \tensor{v^*}{I} M_2(Z) \tensor{v}{I},
\]
where $M_2(Z)$ is a type $2$ structured resolvent as given in Definition \ref{type2res} with associated vector space $\hilbert$, and $v$ is a vector in $\hilbert$.
\end{description}
\end{definition}

In \cite[Section 6]{aty13}, one of the authors with Agler and Young discussed the connections between asymptotic behavior of a  function $h$ in $\mathcal L_n$ along the upper imaginary polyaxis and the existence of representations of a given type for $h$. It turns out that those results lift to results about functions in the free Pick class $\Pick_d$. In fact, the structure of free Pick functions is determined by their behavior on the first level of $\Pi^d$, that is, $d$-tuples of complex numbers. In the following discussion, we will denote by 
	$$\chi = (1, \ldots, 1)$$ 
the element of $\M^d_1$ consisting of all ones. When evaluating a function $h$ on the ray 
$$is\chi = (is, is, \ldots, is) \in \M_1^d,$$ 
we will make the identification $\tensor{\C}{\M_1} \cong \C$. Every $h \in \Pick_d$ has a representation of type $4$.
We show the following analogue of Theorem \ref{oldnevreps}.
\begin{theorem} \label{newnevrep}
The following criteria characterize representations for functions in the free Pick class $\Pick_d.$
Let $h$ be a function defined on $\Pi^d$.
	\begin{description}
		\item[Type 4] \label{type4rep} The following are equivalent:
			\begin{enumerate}
				\item The function $h$ has a representation of type $4$.
				\item The function $h$ is a free Pick function.
			\end{enumerate}
		\item[Type 3] \label{type3asym} The following are equivalent:
			\begin{enumerate}
				\item The function $h$ has a Nevanlinna representation of type $3$;
				\item The function $h$ is a Pick function such that
				\beq
					\liminf_{s \to \infty} \frac{1}{s}\IM h(is\chi) = 0;
				\eeq
				\item The function $h$ is a Pick function such that
				\beq
				\lim_{s \to \infty} \frac{1}{s} \IM h(is\chi) = 0.
				\eeq
			\end{enumerate}
		\item[Type 2] \label{type2asym} The following are equivalent:
			\begin{enumerate}
				\item The function $h$ has a Nevanlinna representation of type $2$;
				\item The function $h$ is a Pick function such that
				\beq
					\liminf_{s \to \infty} s \IM h(is\chi) \leq \infty.
				\eeq
				\item The function $h$ is a Pick function such that
				\beq
				\lim_{s \to \infty} s \IM h(is\chi)
				\eeq
				exists and
				\beq
				\lim_{s \to \infty} s \IM h(is\chi) \leq \infty.
				\eeq
			\end{enumerate}
		\item[Type 1] \label{type1asym} The following are equivalent:
							\begin{enumerate}
				\item The function $h$ has a Nevanlinna representation of type $1$;
				\item The function $h$ is a Pick function such that
				\beq
					\liminf_{s \to \infty} |s h(is\chi)| \leq \infty.
				\eeq
				\item The function $h$ is a Pick function such that
				\beq
				\lim_{s \to \infty} |s h(is\chi)|
				\eeq
				exists and
				\beq
				\lim_{s \to \infty} |s h(is\chi)| \leq \infty.
				\eeq
			\end{enumerate}
	\end{description}
\end{theorem}

The representations have some geometric relationships. Type $3$ representations are the restrictions of type $4$ representations onto subspaces. Type $1$ representations are a special case of type $2$. We can complete the hierarchical description by considering the connection between representations of type $3$ and type $2$. Proofs of the following propositions use the same arguments as in Propositions 5.3 and 5.5 in \cite{aty13}.

\begin{proposition} \label{type32}
If a free Pick function $h$ has a type $2$ representation, then it has  type $3$ representation. Conversely, if a free Pick function $h$ has a type $3$ representation and in addition the vector $v \in \mathcal D(A)$, then $h$ has a type $2$ representation. 
\end{proposition}

We now begin the proof of Theorem \ref{type4rep}. We essentially follow \cite{aty13}.

\subsection{Proof of Theorem \ref{type4rep}}
\subsubsection{Type 4}
\begin{proof}[Proof of \ref{type4rep}: Type 4]
We use a Cayley transform to connect the free Pick class to the free Herglotz class. The Cayley transform between the disc and the right halfplane is given by
\[
z = i \frac{1+\la}{1 - \la}, \, \, \la = \frac{z-i}{z+i}, 
\]
for $z \in \Pi$, $\la \in \D$. For a given Pick function $f$, define a Herglotz function $h$ for $X \in \B^d$ by
\[
h(X) = -if(Z),
\]
where $Z$ is the coordinate-wise Cayley transform of $X$, i.e.
\[
Z_i = i (I - X_i)\inv(I + X_i).
\]

Let $f$ be a free Pick function, and let $h$ be the associated Herglotz function. Then by Corollary \ref{altherg},
\begin{align*}
f(Z) &= ih(X) \\
&= \tensor{a}{I} + i \tensor{v\ad}{I}\left(\tensor{L}{I} - \delta_P(X)\right)\inv\left(\tensor{L}{I} + \delta_P(X)\right)\tensor{v}{I} \\
&= \tensor{a}{I} + i \tensor{v^*}{I}\left(\tensor{L}{I} - \delta_P((Z +i)\inv(Z - i))\right)\inv \times \\
& \hspace{1.5in} \left(\tensor{L}{I} + \delta_P((Z +i)\inv(Z - i))\right)\tensor{v}{I} \\
&= \tensor{a}{I} + i \tensor{v^*}{I}\left(\tensor{L}{I} - \delta_P((Z +i)\inv)\delta_P((Z - i))\right)\inv \times \\
& \hspace{1.5in}\left(\tensor{L}{I} + \delta_P((Z +i)\inv)\delta_P((Z - i))\right)\tensor{v}{I},
\end{align*}
where 
	$$\delta_P(Z) = \sum_i \tensor{P_i}{Z_i}.$$
Now let $M(Z)$ be the expression
\begin{align*}
&M(Z) = i\left(\tensor{L}{I} - \delta_P((Z +i)\inv)\delta_P((Z - i))\right)\inv \times \\
& \hspace{1.5in} \left(\tensor{L}{I} + \delta_P((Z +i)\inv)\delta_P((Z - i))\right).
\end{align*}
With the notation $P_i = P_{\hilbert_i}$, the operator $M(Z)$ can be written
\begin{align*}
&M(Z) = i\left[\delta_P((Z+i)\inv)\left(\delta_P(Z + i)\tensor{L}{I} -\delta_P((Z - i))\right)\right]\inv  \times \\ &\hspace{1in}\delta_P((Z+i)\inv)\left(\delta_P(T + i)\tensor{L}{I} + \delta_P((Z - i))\right) \\
&= i\left(\delta_P(Z + i)\tensor{L}{I} -\delta_P((Z - i))\right)\inv \left(\delta_P(Z + i)\tensor{L}{I} + \delta_P((Z - i))\right) \\
&= i\left(\sum_i\tensor{P_i}{Z_i + i}\sum_i \tensor{P_i L}{I} -\sum_i\tensor{P_i}{Z_i- i}\right)\inv \times \\
&= \hspace{1.5in} \left(\sum_i\tensor{P_i}{Z_i + i}\sum_i\tensor{P_i L}{I} + \sum_i\tensor{P_i}{Z_i - i}\right) \\
&= i\left(\sum_i\tensor{P_i L}{Z_i + i} -\sum_i\tensor{P_i}{Z_i- i}\right)\inv \left(\sum_i\tensor{P_i L}{Z_i + i}+ \sum_i\tensor{P_i}{Z_i - i}\right) \\
&= i\left(\sum_i\tensor{P_i (L - I)}{Z_i} -\sum_i\tensor{P_i(L + I)}{i}\right)\inv \left(\sum_i\tensor{P_i (L + I)}{Z_i}+ \sum_i\tensor{P_i (L - I)}{i}\right) \\
&= i\left(\delta_P(Z)\tensor{L - I}{I} -i\tensor{L + I}{I}\right)\inv \left(\delta_P(Z)\tensor{L + I}{I}+ i\tensor{L - I}{I}\right). \\
\end{align*}

Let $\N = \ker(I - L)$. Decompose $L$ according to $\hilbert = \N \oplus \K$, where $\K = \N^\perp$, so that 
\[
L = \begin{bmatrix} I & 0 \\ 0 & L_0 \end{bmatrix} \begin{array}{c} \N \\ \K \end{array}
\]
where $L_0$ is unitary and $\ker(I - L_0) = \{0\}$.

Then
\begin{align}
M(Z) &= i\left(\delta_P(Z)\tensor{\bbm 0 & 0 \\ 0 & L_0 - I \ebm}{I} + \tensor{i\bbm 2 & 0 \\ 0 & L_0 + I \ebm}{I}\right)\inv \times \notag \\
& \hspace{1in} \left(\delta_P(Z)\tensor{\bbm 2 & 0 \\ 0 & I + L_0 \ebm}{I}+ \tensor{i\bbm 0 & 0 \\ 0 & L_0 + I \ebm}{I}\right) \notag \\
&= \left(-\delta_P(Z)\tensor{\bbm 0 & 0 \\ 0 & I- L_0 \ebm}{I} + \tensor{\bbm 2i & 0 \\ 0 & i(I+L_0) \ebm}{I}\right)\inv \notag \times \\ &\hspace{1in} \left(\delta_P(Z)\tensor{\bbm 2i & 0 \\ 0 & i(I + L_0) \ebm}{I}+ \tensor{\bbm 0 & 0 \\ 0 & I - L_0 \ebm}{I}\right). \label{prefactor}
\end{align}
We would like to continue by writing
\begin{align}
M(Z) &= \tensor{\bbm -\frac{1}{2}i & 0 \\ 0 & (I - L_0)\inv \ebm}{I}\left(-\delta_P(Z)\tensor{\bbm 0 & 0 \\ 0 & I \ebm}{I} + \tensor{\bbm I & 0 \\ 0 & i\frac{I + L_0}{I - L_0} \ebm}{I}\right)\inv \times \notag \\ 
&\hspace{1in} \left(\delta_P(Z) \tensor{\bbm I & 0 \\ 0 & i\frac{I+L_0}{I - L_0} \ebm}{I} + \tensor{\bbm 0 & 0 \\ 0& I \ebm}{I}\right)\tensor{\bbm 2i & 0 \\ 0 & I- L_0 \ebm}{I}, \label{postfactor}
\end{align}
but to do so, we need to show that the unbounded, partially defined operator above makes sense. Let 
\[
A = i\frac{I + L_0}{I - L_0}.
\]
$A$ is self-adjoint and densely defined on $\K$ as $L_0$ is unitary on $\K$ and $\ker (1-L_0) = \{0\}$  \cite[Section 22]{rienag}. Let $\mathcal D(A)$ be the domain of $A$. Then $\mathcal D(A)$ is the dense subspace $\ran(I - L_0)$ of $\K$. Then by the definition of $A$, 
\[
(I - L_0)\inv = \frac{1}{2}(I - iA),
\]
which is an equation between bijective operators from $\mathcal D(A) \to \M$. Likewise, the equation
\[
I + L_0 = -2iA(I - iA)\inv
\]
relates bounded operators from $\mathcal D(A) \to \M$.

We will now justify factoring the expression in \eqref{prefactor}. Since $\ker(I - L_0) = \{0\}$, 
\[
\bbm 2i & 0 \\ 0 & I - L_0 \ebm
\]
is a bijection between $\hilbert$ and $\N \oplus \mathcal D(A)$. Now decompose the $P_i$ with respect to $\hilbert = \N \oplus \K$, so that 
for each $i = 1, \ldots, d$,
\[
P_i = \bbm X_i & B_i \\ B_i\ad & Y_i \ebm,
\]
which allows us to write $\delta_P(Z)$ as
\[
\delta_P(Z) = \sum_{i = 1}^{d} \tensor{P_i}{Z_i} = \sum_i \tensor{\bbm X_i & B_i \\ B_i\ad & Y_i \ebm}{Z_i}.
\]
Then 
\begin{align}
 &\left( -\delta_P(Z) \tensor{\bbm 0 & 0 \\ 0 & I \ebm}{I} + \tensor{\bbm I & 0 \\ 0 & A \ebm}{I}\right)\inv =  \left[-\left(\sum_i\tensor{P_i}{Z_i}\right)\tensor{\bbm 0 & 0 \\ 0 & I \ebm}{I} + \tensor{\bbm I & 0 \\ 0 & A \ebm}{I}\right]\inv \notag \\
&= \left[-\left(\sum_i\tensor{\bbm X_i & B_i \\ B_i\ad & Y_i \ebm}{Z_i}\right)\tensor{\bbm 0 & 0 \\ 0 & I \ebm}{I} + \tensor{\bbm I & 0 \\ 0 & A \ebm}{I}\right]\inv \notag \\
&= \left[-\left(\sum_i \tensor{\bbm 0 & B_i \\ 0 & Y_i \ebm}{Z_i}\right) + \tensor{\bbm I & 0 \\ 0 & A \ebm}{I}\right]\inv \notag \\ 
&= \bbm I & -\sum_i\tensor{B_i}{Z_i} \\ 0 & \tensor{A}{I} - \sum_i\tensor{Y_i}{Z_i} \ebm\inv \notag \\
&= \bbm I & -\sum_i\tensor{B_i}{Z_i} \left(\tensor{A}{I} - \delta_Y(Z)\right)\inv \\ 0 & \left(\tensor{A}{I} - \delta_Y(Z)\right)\inv \ebm. \label{leftfactorinverse}
\end{align}
That $g(Z) = \left(\tensor{A}{I} - \delta_Y(Z)\right)\inv$ is a well-defined function bounded on $\mathcal D(A)$ follows by an argument similar to that in \cite{aty13}. Thus 
\begin{align}
&\left( -\delta_P(Z) \tensor{\bbm 0 & 0 \\ 0 & I \ebm}{I} + \tensor{\bbm I & 0 \\ 0 & A \ebm}{I}\right)\inv \notag \\
&= \tensor{\bbm -\frac{1}{2}i & 0 \\ 0 & (I - L_0)\inv \ebm}{I}\left(-\delta_P(Z)\tensor{\bbm 0 & 0 \\ 0 & I \ebm}{I} + \tensor{\bbm I & 0 \\ 0 & i\frac{I + L_0}{I - L_0} \ebm}{I}\right)\inv \notag \\
&= \tensor{\bbm -\frac{1}{2}i & 0 \\ 0 & \frac{1}{2}(I - iA) \ebm}{I}\left(\tensor{\bbm I & 0 \\ 0 & A \ebm}{I} - \delta_P(Z)\tensor{\bbm 0 & 0 \\ 0 & I \ebm}{I} \right)\inv \label{left}
\end{align}
is also a bijection from $\N \oplus \mathcal D(A) \to \hilbert$, and so we can apply inverses to the left factor of the right-hand side of \eqref{prefactor}. Similar reasoning allows us to conclude that 
\begin{align} 
&\left(\delta_P(Z)\tensor{\bbm 2i & 0 \\ 0 & i(I + L_0) \ebm}{I}+ \tensor{\bbm 0 & 0 \\ 0 & I - L_0 \ebm}{I}\right)  \notag \\
&= \left(\delta(Z) \tensor{\bbm I & 0 \\ 0 & i\frac{I+L_0}{I - L_0} \ebm}{I} + \tensor{\bbm 0 & 0 \\ 0& I \ebm}{I}\right)\tensor{\bbm 2i & 0 \\ 0 & I- L_0 \ebm}{I} \notag \\
&= \left(\delta_P(Z) \tensor{\bbm I & 0 \\ 0 & A \ebm}{I} + \tensor{\bbm 0 & 0 \\ 0& I \ebm}{I}\right)\tensor{\bbm -\frac{1}{2}i & 0 \\ 0 & \frac{1}{2}(I - iA) \ebm\inv}{I} \label{right}
\end{align}
as operators on $\hilbert$. Thus, upon combining \eqref{prefactor}, \eqref{left}, and \eqref{right} and pre- and post-multiplying by $2$ and $\frac{1}{2}$, we have established 
\begin{align}
M(Z) &= \tensor{\bbm -i & 0 \\ 0 & I - iA \ebm}{I}\left( \tensor{\bbm I & 0 \\ 0 & A \ebm}{I} -\delta_P(Z)\tensor{\bbm 0 & 0 \\ 0 & I \ebm}{I} \right)\inv \times \notag \\ 
&\hspace{1in} \left(\delta_P(Z) \tensor{\bbm I & 0 \\ 0 & A \ebm}{I} + \tensor{\bbm 0 & 0 \\ 0& I \ebm}{I}\right)\tensor{\bbm -i & 0 \\ 0 & I-iA \ebm\inv}{I}.
\end{align}
Thus, we have shown that $M$ is a type $4$ resolvent, and therefore that $h$ has a Nevanlinna representation of type $4$, i.e.
\[
h(Z) = \tensor{a}{I} + \tensor{v\ad}{I} M(Z) \tensor{v}{I}.
\] 

The converse follows from the fact that the imaginary part of $M_4$ is positive as was remarked in Definition \ref{type4res}.
\end{proof}

\subsubsection{Type 3}
\begin{proof}[Proof of Theorem \ref{type3asym}: Type 3]
(1) $\Rightarrow$ (3): Follows from Theorem \ref{oldnevreps}.

(3) $\Rightarrow$ (2) is trivial.

(2) $\Rightarrow$ (1): Suppose that condition (2) holds, that is $h$ is a free Pick function with
\[
\liminf_{s\to \infty} \IM \frac{1}{s} h(is\chi) = 0.
\]
As a Pick function, $h$ has a type $4$ Nevanlinna representation, that is there exist $a \in \R, \hilbert, \N \subset \hilbert,$ operators $A, Y$ on $\N^\perp$, an orthogonal decomposition $P$ of $\hilbert$ and a vector $v\in\hilbert$ such that 
\[
\tensor{a}{I} + \tensor{ v\ad}{I} M(Z) \tensor{v}{I},
\]
where
\begin{align*}
M(Z) &= \tensor{\bbm -i & 0 \\ 0 & 1 - iA \ebm}{I}\left( \tensor{\bbm 1 & 0 \\ 0 & A \ebm}{I} -\delta_P(Z)\tensor{\bbm 0 & 0 \\ 0 & 1 \ebm}{I} \right)\inv \times \notag \\ 
&\hspace{1in} \left(\delta_P(Z) \tensor{\bbm 1 & 0 \\ 0 & A \ebm}{I} + \tensor{\bbm 0 & 0 \\ 0& 1 \ebm}{I}\right)\tensor{\bbm -i & 0 \\ 0 & 1-iA \ebm\inv}{I},
\end{align*}
For $Z = is\chi$,
\[
\delta_P(Z) = \tensor{is}{1_\C},
\]
so for $s>0$, $M(is\chi)$ becomes
\begin{align*}
M(is\chi) &=  \tensor{\bbm -i & 0 \\ 0 & 1 - iA \ebm}{1}\left( \tensor{\bbm 1 & 0 \\ 0 & A \ebm}{1} -\tensor{is}{1}\tensor{\bbm 0 & 0 \\ 0 & 1 \ebm}{1} \right)\inv \times \notag \\ 
&\hspace{1in} \left(\tensor{is}{1} \tensor{\bbm 1 & 0 \\ 0 & A \ebm}{1} + \tensor{\bbm 0 & 0 \\ 0& 1 \ebm}{1}\right)\tensor{\bbm -i & 0 \\ 0 & 1-iA \ebm\inv}{1} \\
&= \bbm -i & 0 \\ 0 & 1 - iA \ebm\left( \bbm 1 & 0 \\ 0 & (A - is)\inv \ebm\right)  \left( \bbm is & 0 \\ 0 & 1 + is A \ebm\right)\bbm -i & 0 \\ 0 & 1-iA \ebm\inv \\
&= \bbm is & 0 \\ 0 & (1-iA)(A - is)\inv(1+isA)(1-iA)\inv \ebm.
\end{align*}
Now let $v_1 = P_\N v$ and $v_2 = P_{\N^\perp} v$. Then
\begin{align*}
&h(is\chi) = \tensor{a}{1} + \tensor{v\ad}{1} M(is\chi) \tensor{v}{1} \\
&= a + \bbm is & 0 \\ 0 & (1-iA)(A - is)\inv(1+isA)(1-iA)\inv \ebm \bpm v_1 \\ v_2 \epm \\
&= a + is v_1\ad  v_1 + v_2\ad(1-iA)(A - is)\inv(1+isA)(1-iA)\inv v_2 \\
&= a + is\norm{v_1}^2 + \ip{(1-iA)(A-is)\inv(1 + isA)(1 - iA)\inv v_2}{v_2}_{\N^\perp}.
\end{align*}
To compute $1/s \IM h(is\chi)$, we find
\begin{align*}
\frac{1}{s} \IM &\left(a + is\norm{v_1}^2 + \ip{(1-iA)(A - is)\inv(1+isA)(1-iA)\inv v_2}{v_2}\right) \\
&= \norm{v_1}^2 + \frac{1}{s}\IM \ip{(1-iA)(A - is)\inv(1+isA)(1-iA)\inv v_2}{v_2}\\
&\geq \norm{v_1}^2
\end{align*}
by Corollary 2.7 in \cite{aty13}. By hypothesis,
\begin{align*}
0 &= \liminf_{s \to \infty} \frac{1}{s} h(is\chi) \\
&= \liminf_{s \to \infty} \norm{v_1}^2 + \frac{1}{s}\IM \ip{(1-iA)(A - is)\inv(1+isA)(1-iA)\inv v_2}{v_2} \\ &\geq \norm{v_1}^2,
\end{align*}
and so $v_1 = 0$. We claim that with the Hilbert space $\N^\perp$, the positive decomposition $Y$ of $\N^\perp$ given by the compression of the orthogonal decomposition $P$ to $\N^\perp$, the operator $A$ on $\N^\perp$, the vector $v_2 \in \N^\perp$, and the real number $a$, we get that the compression of the type $4$ representation to $\tensor{\N^\perp}{\M}$ is a type $3$ representation for $h$.  Recall that $(\tensor{A}{I} - \delta_Y(Z))\inv$ is a well-defined, bounded operator on $\Pi^d$. Then
\begin{align*}
h(Z) &= \tensor{a}{I} + \tensor{v\ad}{I} M(Z) \tensor{v}{I} \\
&= \tensor{a}{I} + \tensor{ v\ad}{I} \tensor{\bbm -i & 0 \\ 0 & 1 - iA \ebm}{I}\left( \tensor{\bbm 1 & 0 \\ 0 & A \ebm}{I} -\delta_P(Z)\tensor{\bbm 0 & 0 \\ 0 & 1 \ebm}{I} \right)\inv \times \notag \\ 
&\hspace{1in} \left(\delta_P(Z) \tensor{\bbm 1 & 0 \\ 0 & A \ebm}{I} + \tensor{\bbm 0 & 0 \\ 0& 1 \ebm}{I}\right)\tensor{\bbm -i & 0 \\ 0 & 1-iA \ebm\inv}{I} \tensor{v}{I} \\
&= \tensor{a}{I} + \tensor{ v\ad}{I} \tensor{\bbm -i & 0 \\ 0 & 1 - iA \ebm}{I}\left( \tensor{\bbm 1 & 0 \\ 0 & A \ebm}{I} -\sum \tensor{\bbm 0 & B_i \\ 0 & Y_i \ebm}{Z^i} \right)\inv \times \notag \\ 
&\hspace{1in} \left(\sum \tensor{\bbm X_i & B_i A \\ B_i\ad & Y_i A \ebm}{Z^i} + \tensor{\bbm 0 & 0 \\ 0& 1 \ebm}{I}\right)\tensor{\bbm i & 0 \\ 0 & (1-iA)\inv \ebm}{I} \tensor{v}{I} \\
&= \tensor{a}{I} + \tensor{ v\ad}{I} \bbm \tensor{-i}{I} & 0 \\ 0 & \tensor{1 - iA}{I} \ebm \bbm 1 & \delta_B(Z) \left(\tensor{A}{I} - \delta_Y(Z)\right)\inv \\ 0 & \left(\tensor{A}{I} - \delta_Y(Z)\right)\inv \ebm \times \notag \\ 
&\hspace{1in} \bbm \delta_X(Z) & \delta_B(Z)\tensor{A}{I} \\ \delta_{B\ad}(Z) & \tensor{I}{I} + \delta_Y(Z)\tensor{A}{I} \ebm \bbm \tensor{i}{I} & 0 \\ 0 & \tensor{(1-iA)\inv}{I} \ebm \tensor{v}{I} \\
&= \tensor{a}{I} + \tensor{ v\ad}{I} \bbm \tensor{-i}{I} & -i\delta_B(Z)\left(\tensor{A}{I} - \delta_Y(Z)\right)\inv \\ 0 & \tensor{1 - iA}{I}\left(\tensor{A}{I} - \delta_Y(Z)\right)\inv \ebm \times \\
&\hspace{1in} \bbm i\delta_X(Z) & \delta_B(Z)\tensor{A(1- iA)\inv}{I} \\ i\delta_{B\ad}(Z) & \left(\tensor{I}{I} + \delta_Y(Z)\tensor{A}{I}\right)\tensor{(1 - iA)\inv}{I} \ebm \tensor{v}{I}. \\
\end{align*}
Since $v_1 = 0$, we can compress the rather unwieldy operator resulting from the multiplication to $\tensor{\N^\perp}{\M}$, which gives
\[
h(Z) = \tensor{a}{I} + \tensor{v_2\ad(1-iA)}{I}\left(\tensor{A}{I} - \delta_Y(Z)\right)\inv\left(\tensor{I}{I} + \delta_Y(Z)\tensor{A}{I}\right)\tensor{(1-iA)\inv v_2 }{I},
\]
that is, $h$ has a type $3$ Nevanlinna representation.
\end{proof}

\subsubsection{Type 2}
\begin{proof}[Proof of Theorem \ref{type2asym}: Type 2]
(1) $\Rightarrow$ (2): Follows from Theorem \ref{oldnevreps}.

(3) $\Rightarrow$ (2) is trivial.

(2) $\Rightarrow$ (1): Suppose that for $h \in \Pick_d$, 
\beq \label{cond2}
\liminf_{s \to \infty} s \IM h(is\chi) < \infty.
\eeq
This obviously implies that
\[
\liminf_{s \to \infty} \frac{1}{s} \IM h(is\chi) = 0,
\]
and so by Theorem \ref{type3asym} Type 3, there exist $a, \hilbert, Y, A,$ and $v$ so that $h$ has a type $3$ representation. As
\[
\delta_Y(is\chi) = \tensor{is}{1_\C},
\]
\begin{align*}
h(is\chi) &= a + \alpha\ad(1 - iA)(A- is)\inv(1+isA)(1-iA)\inv\alpha \\
&= a + \ip{(1 - iA)(A - is)\inv(1 + isA)(1 - iA)\inv v}{v}_\hilbert.
\end{align*}

Let $\nu_{v, v} = \nu$ be the scalar spectral measure for $A$. Then for $s > 0$,
\begin{align*}
s\,\IM h(is\chi) &= s\,\IM \int \frac{1 + ist}{t - si} \,\,\dd\nu(t) \\
&= \int \frac{s^2(1 + t^2)}{t - is}\,\,\dd\nu(t).
\end{align*}
As $s \to \infty,$ the integrand increases monotonically to $1 + t^2$. Then by \eqref{cond2}
\[
\int (1+t^2)\,\,\dd\nu(t) < \infty,
\]
and so by the Spectral Theorem
\[
\ip{(1+A^2)v}{v} < \infty,
\]
which gives $v \in \mathcal D(A)$. Therefore, by Theorem \ref{type32}, $h$ has a type $2$ representation.
\end{proof}

\subsubsection{Type 1}
\begin{proof}[Proof of \ref{type1asym} Type 1]

(1) $\Rightarrow$ (3) follows from \ref{oldnevreps}.

(3) $\Rightarrow$ (2) is trivial.

(2) $\Rightarrow$ (1): Suppose that 
\beq \label{type1ineq}
\liminf_{s \to infty} s \abs{h(is\chi)} < \infty.
\eeq
As 
\[
\liminf_{s \to \infty} s\,\IM h(is\chi) \leq \liminf{s \to \infty} s\abs{h(is\chi)},
\]
by Theorem \ref{type2asym} Type 2, $h$ has a Nevanlinna representation of type $2$, that is there exist $\hilbert, Y, A,$ and $\alpha \in \mathcal D(A)$ such that
\[
h(Z) = \tensor{a}{I} + \tensor{\alpha\ad}{I}\left(\tensor{A}{I} - \delta_Y(Z)\right)\inv\tensor{\alpha}{I}.
\]
It remains to show that $a = 0$. For $Z = is\chi$, 
\[
\delta_Y(is\chi) = \tensor{is}{1_\C}.
\]
Then 
\[
h(is\chi) = a + \ip{(A - is)\inv\alpha}{\alpha}_\hilbert.
\]
By \eqref{type1ineq}, there must exist a sequence $s_n \to \infty$ such that $h(is_n\chi) \to 0$. But on this sequence, 
\[
\RE h(is_n\chi) = a + \ip{A(A^2 + s_n^2)\inv\alpha}{\alpha} \to a,
\]
and therefore it must be the case that $a = 0$. Thus $h$ has a type $1$ representation.

\end{proof}

%\section{Appendix}

\bibliography{references}
\bibliographystyle{plain}

\end{document}